\documentclass[11pt]{article}

\usepackage{amssymb,amsthm,amsmath,hyperref}
\usepackage[english]{babel}
\usepackage{indentfirst}
\usepackage{graphicx,float}
\usepackage{xcolor,txfonts}
\numberwithin{equation}{section}

\newtheorem{thm}{Theorem}[section]
\newtheorem{lem}{Lemma}[section]

\newtheorem{prop}{Proposition}[section]
\newtheorem{rem}{Remark}[section]

\theoremstyle{definition}
\theoremstyle{remark}

\def\br{\mathbf R}

\def\p{\partial}

\def\ri{{\rm i}}

\def\dd{{\rm d}}
\def\di{{\rm div}}

\allowdisplaybreaks

\begin{document}
\title{\bf On decay of solutions to the anisotropic Boussinesq equations near the hydrostatic balance in half space   $\br_+^3$}
\author{Wanrong  Yang$^a$\thanks{E-mail: yangwanrong1618@163.com},
Aibin Zang$^b$ \thanks{E-mail: abzang@jxycu.edu.cn}\thanks{Corresponding author} \\
\textit{\small  a. School of Mathematics and Information Sciences, North Minzu University,   Yinchuan, Ningxia 750021, PR China} \\
\textit{\small  b. School of Mathematics and Computer Science,    Yichun University, Yichun, Jiangxi, P. R. China}
}

\date{}
\maketitle

\begin{abstract}
	The system of the Boussinesq equations is one of the most important models for geophysical fluids. This paper focuses on the initial-boundary problem of the 3D incompressible anisotropic Boussinesq system with horizontal dissipation. The goal here is to assess the stability property and large-time behavior of perturbations near the hydrostatic balance. By utilizing the structure of the system, the energy methods and the means of bootstrapping argument, we prove the global stability property in the Sobolev space $H^3(\br^3_+)$. After taking a Fourier transform in $x_h = (x_1, x_2)$ and Fourier cosine and sine transforms in $x_3$  for the system, we obtain the  decay rates for the global solution itself as well as its derivatives.
\vspace{4mm}


\end{abstract}

\section{Introduction}\label{S0}

The Boussinesq system arises from a zeroth order approximation of the coupling between Navier-Stokes equations and the thermodynamic equations. It can be used as a model to describe many geophysical phenomena \cite{Go,wx5,Ro}. The standard 3D incompressible Boussinesq equations are given
by
\begin{eqnarray}
\begin{cases}
u_t+u\cdot\nabla u+\nabla P=\nu\Delta u+\Theta e_3, &x\in\br^3,t>0,\\
\Theta_t+u\cdot\nabla\Theta=\kappa\Delta\Theta,&x\in\br^3,t>0,\\
\nabla\cdot u=0,&x\in\br^3,t>0,
\end{cases}\label{E1.1}
\end{eqnarray}
where $u=(u_1,u_2,u_3)$, $P$, $\Theta$ respectively, denote the velocity, pressure, and temperature. $e_3=(0,0,1)$ and $\nu\ge 0,\kappa \ge 0$ are the viscosity and the thermal diffusivity.

There are many results for the Cauchy problem of the 3D Boussinesq equations with initial data
\begin{equation*}\label{E1.2}
u(x,t=0)=u_0(x), \Theta(x,t=0)=\Theta_0(x).
\end{equation*}
When $\nu$ and $\kappa$ are positive constants, the local well-posedness can be easily established for the 3D Boussinesq system by the energy method. One can refer \cite{wx10},\cite{wx12},\cite{wx18}-\cite{wx15},\cite{wx21} and reference therein. R. Danchin and M. Paicu \cite{DP} proved the global existence of weak solution for $L^2$ data and the global well-posedness for small smooth data. T. Hmidi and F. Rousset \cite{HR1,HR2} proved the global well-posedness of the 3-D axisymmetric Boussinesq system without swirl. Abidi, Hmidi and Keraani \cite{wx34} also showed the global well-posedness for the Boussinesq system with axisymmetric initial data. Jiu, Wang and Wu \cite{wx6} established partial regularity for the appropriate weak solution by the De Giorgi iterative approach. In recent years, there are many works devoted to the study of the Boussinesq system with partial dissipation. Under the assumption that the initial data are axisymmetric without swirl, Miao and Zheng \cite{wx11} proved the global well-posedness for the 3D Boussinesq equation with horizontal dissipation. Recently, big progresses have been made on the stability problem. Dong \cite{wx25} studied asymptotic stability to the 3D Boussinesq equation in the whole space with a velocity damping term. In addition, the decay rates of the velocity and large-time behavior of the temperature also were given. Wu and Zhang \cite{wx33} solved the stability and large-time behavior problem with mixed partial dissipation in spatial domain $\Omega=\br^2\times T$ with  $T=[-\frac{1}{2},\frac{1}{2}]$. Shang and Xu \cite{wx14} examined the stability and the decay of the corresponding linearized systems of 3D Boussinesq equations with horizontal dissipation. Ji, Yan and Wu \cite{JYWu} further expanded the results and obtained the optimal decay for the corresponding nonlinear Boussinesq system.

We try to consider this system in the bounded domain $\Omega$ with smooth boundary $\partial\Omega$. For $\Omega\subset\br^2,$ there are many literatures to consider the case with either only viscosity, or only thermal diffusivity. M. Lai, R. Pan and  K. Zhao \cite{LPZ} have obtained the global existence results of the 2D Boussinesq system without thermal diffusivity and the velocity satisfying Dirichlet boundary conditions. W. Hu, I. Kukavica and M. Ziane in \cite{HKZ} improved this result under less regularity of initial data. Y. Sun and Z. Zhang studied the global regularity for the 2D Boussinesq system with viscosity and thermal diffusivity satisfying Dirichlet boundary conditions. In \cite{Hwuphyd}, W. Hu, Y. Wang, J. Wu, B. Xiao and J. Yuan investigated the 2D Boussinesq system without thermal diffusivity satisfying the following Navier boundary conditions
\begin{equation}\label{B1.4}
u\cdot n=0, \,\, 2n\cdot D(u)\cdot\tau+\alpha u\cdot\tau=0,
\end{equation}
where $D(u)$ is the Cauchy stress of the velocity, and obtained a unique global strong solution for this initial-boundary value problem. However, for $\Omega\subset\br^3$ only local well-posedness results have been established, whether smooth solution blows up in finite time remains open.

In this paper, we would consider the global stability for the Boussinesq equations in $\br_+^3.$ We impose Navier slip boundary conditions for velocity,
\begin{equation}\label{N1.5}
u_3=0,\,\,\, \partial_3 u_1=\partial_3 u_2=0, ~\mbox{on}~x_3=0,
\end{equation}
which is a special case of the boundary conditions \eqref{B1.4} and the no-slip boundary condition for temperature
\begin{equation}\label{N1.6}
\Theta=0, ~\mbox{on}~x_3=0.
\end{equation}

The hydrostatic equilibrium given by
\begin{equation}\label{hyb1.7}
u^{(0)}=(0,0,0),\,\, \Theta^{(0)}=x_3,\,\, P^{(0)}=\frac{1}{2}x_3^2,
\end{equation}
is a very special steady-state solution of \eqref{E1.1} satisfying the boundary conditions \eqref{N1.5} and \eqref{N1.6}. To clarify the global stability problem near the hydrostatic balance \eqref{hyb1.7} with the boundary conditions \eqref{N1.5} and \eqref{N1.6} in half space, we consider the Boussinesq system with horizontal dissipations governing the perturbation  $(u,\theta,p)$ with $\theta = \Theta-\Theta^{(0)}, ~p = P-P^{(0)},$
 \begin{eqnarray}
\begin{cases}
u_t+u\cdot\nabla u+\nabla p=\nu\Delta_h u+\theta e_3, &x\in\br^3_+,t>0,\\
\theta_t+u\cdot\nabla\theta+u_3=\kappa\Delta_h\theta,&x\in\br^3_+,t>0,\\
\nabla\cdot u=0,&x\in\br^3_+,t>0,\\
u(x,0)=u_0(x), \theta(x,0)=\theta_0(x), &x\in\br_+^3,
\end{cases}\label{E1.9}
\end{eqnarray}
here $\Delta_h=\partial_{x_1}^2+\partial_{x_2}^2$ stands for the horizontal Laplacian.  For brevity, let $\nu=\kappa=1$ hereafter. In this paper, we shall check the following results.

\begin{thm}\label{T1.1}(Global stability). Assume that initial data $(u_0,\theta_0)\in H^3(\br_+^3)$ with boundary conditions \eqref{N1.5} and \eqref{N1.6}, satisfies $\nabla\cdot u_0=0.$ In addition, let $\p_3^2\theta_0=0, ~\mbox{on}~ x_3=0.$ Then there exists $\varepsilon>0$ such that, if
$$\|u_0\|_{H^3}+\|\theta_0\|_{H^3}\leqslant\varepsilon,$$
then the system \eqref{E1.9} with boundary conditions \eqref{N1.5} and \eqref{N1.6} has a unique global solution $(u,\theta)\in L^\infty(0,\infty;(H^3(\br_+^3)))$ satisfying, for a constant $C>0$ and for all $t>0,$
\begin{subequations}
\begin{align*}
\|u\|^2_{H^3}+\|\theta\|^2_{H^3}+\int_0^t\left(\|\nabla_h u(\tau)\|^2_{H^3}+\|\nabla_h\theta(\tau)\|^2_{H^3}\right)\dd\tau\leqslant C\varepsilon^2.
\end{align*}
\end{subequations}

\end{thm}


\begin{thm}\label{T1.2}(Decay)  Let $\frac{9}{10}< \sigma<1.$ Assume that $(u_0,\theta_0)\in H^3(\br_+^3)$ satisfy the conditions stated in Theorem \ref{T1.1} and
\begin{eqnarray}
&&\|u_0\|_{H^3 }+\|\theta_0\|_{H^3 }\leqslant\varepsilon,\label{I1.6}\\
&&\|\Lambda_h^{-\sigma} u_0\|_{L^2}+\|\Lambda_h^{-\sigma}\theta_0\|_{L^2}\leqslant\varepsilon,\label{I1.7}\\
&&\|\partial_3\Lambda_h^{-\sigma} u_0\|_{L^2}+\|\partial_3\Lambda_h^{-\sigma}\theta_0\|_{L^2}\leqslant\varepsilon,\label{I1.8}
\end{eqnarray}
for some  sufficiently small $\varepsilon>0.$ then for the initial boundary value problem \eqref{E1.9} with boundary conditions \eqref{N1.5} and \eqref{N1.6}, there exists a unique global solution $(u,\theta)$, which satisfies,
for a constant $C>0$ and for any $  \frac{1}{2}-\frac{\sigma}{2}\leqslant\delta<\frac{\sigma}{8}-\frac{1}{16} $ and all $t\geqslant0,$
\begin{eqnarray*}
&&\|u(t)\|_{H^3 }+\|\theta(t)\|_{H^3 }\leqslant C\varepsilon, \\
&&\|\Lambda_h^{-\sigma} u(t)\|_{L^2}+\|\Lambda_h^{-\sigma}\theta(t)\|_{L^2}\leqslant C\varepsilon, \\
&&\|u(t)\|_{L^2},\|\theta(t)\|_{L^2}\leqslant C\varepsilon (1+t)^{-\frac{\sigma}{2}+\delta},\\
&&\|\partial_3u_h(t)\|_{L^2},\|\partial_3\theta(t)\|_{L^2}\leqslant C\varepsilon (1+t)^{-\frac{\sigma}{2}+3\delta},\\
&&\|\p_3 u_3\|_{L^2}, \|\nabla_h u(t)\|_{L^2},\|\nabla_h\theta(t)\|_{L^2}\leqslant C\varepsilon (1+t)^{-\frac{\sigma+1}{2}+\delta},
\end{eqnarray*}
here we   always denote the operator $\Lambda_h^{-\sigma}$ to be the Fourier multiplier with symbol $|\xi_h|^{-\sigma}$ for  horizontal components of vector $\xi=(\xi_1,\xi_2,\xi_3)=(\xi_h,\xi_3)$.
\end{thm}

\begin{rem} We will find the decay for the linearized equations of \eqref{E1.9} (see Proposition \ref{prop2.1} ) is the same as the Cauchy problem stated in \cite{JYWu}. However,  the large-time behavior of the solutions for the nonlinear initial boundary value problem is quite different from that of the Cauchy problem, where the decay rates depend on the high regularity of the solutions, whereas we only have $\|(u,\theta)\|_{H^3}$ under the boundary conditions \eqref{N1.5} and \eqref{N1.6} in Theorem \ref{T1.1}. In addition, the anisotropic dissipation here makes the decay problem even more challenging, and the decay rates presented in Theorem \ref{T1.2} also reflect the extra smoothing due to the coupling in the system.
\end{rem}

The proof of Theorem \ref{T1.2} is extremely elaborate. We note that direct energy estimates are not adequate. We would like to resort the integral representation of (\ref{E1.9}). First, we take the Fourier transform in $x_h = (x_1, x_2)$ and Fourier cosine and sine transforms in $x_3$ of the linearized equations of (\ref{E1.9}), and obtain the optimal decay rates of the linearized initial-boundary value problem. Then we take advantage of Duhamel's principle to obtain the formula of the solution for the nonlinear system. This form relies on seven kernel functions which are degenerate and anisotropic in the frequency space. Finally, we obtain the decay rates of $(u,\theta)$ and its derivatives via the integral form. The detailed estimates are provided in Section 3.

The rest of this paper is arranged as follows: in Section \ref{S1}, we give the proof of the nonlinear stability result in Theorem \ref{T1.1} .  One presents a solution formula for the linearized problem in terms of Fourier cosine and sine transforms in Section \ref{S2}. The decay rates, our result stated in  Theorem \ref{T1.2} are established  in the last section. To simplify the notation, we shall write ${\| f \|_{{L^p}}}$ for ${\| f \|_{{L^p}(\mathbb{R}^{3}_{+})}}$, ${\| f \|_{L_{{x_i}}^p}}$ for the $L^2$-norm in $x_i$-variable, $\partial_i$ for $\partial_{x_i}(i=1,2,3)$, and $\nabla_h=(\partial_1,\partial_2)$.

\section{Nonlinear stability}\label{S1}

In this section, we prove that the global stability for the initial  problem \eqref{E1.9} with boundary conditions \eqref{N1.5} and \eqref{N1.6} . To verify Theorem \ref{T1.1}, We first introduce several significant tools to be used in the proof. The first lemma is the calculus inequalities in Sobolev spaces(see \cite{wx3}).
\begin{lem}\label{lem000}
	For all $m \in  Z^+ \cup \{0\}$, there exists $c > 0$ such that, for all $u, v \in L^\infty\cap H^m(\br^n_+)$,
	\begin{equation*}
	\begin{aligned}
	\|uv\|_{H^m}\le c(\|u\|_{L^\infty}\|D^m v\|_{L^2}+\|D^m u\|_{L^2}\|v\|_{L^\infty}).
	\end{aligned}
	\end{equation*}
\end{lem}
\begin{proof}
	For any $u\in L^\infty\cap H^m(\br^n_+)$, we define the extension operator $\tilde{u}$ by $$\tilde{u}:=
	\begin{cases}
	u(x', x_n), &x_n\geq 0\\
	\sum\limits_{i = 1}^m {c_i u(x', -\frac{x_n}{i})}, &x_n < 0\\
	\end{cases}
	$$
	where $\sum\limits_{i = 1}^m c_i(-\frac{1}{i})^k =1, \forall k=0, \cdot\cdot\cdot m-1$. Then we have
	$\tilde{u}\in L^\infty\cap H^m(\br^n)$, which satisfies $$\|\tilde{u}\|_{L^\infty\cap H^m(\br^n)}\leq C(n,m)\|u\|_{L^\infty\cap H^m(\br^n_+)}.$$
	By H\"{o}lder's inequality and the Leibnitz differentiation formula, we have
	$$\|D^\alpha(\tilde{u}\tilde{v})\|_{L^2} \leq C_\alpha \sum\limits_{\beta\leq \alpha}\|D^\beta \tilde{u} D^{\alpha-\beta}\tilde{v}\|_{L^2}
	\leq C_\alpha \sum\limits_{\beta\leq \alpha}\|D^\beta \tilde{u}\|_{L^{\frac{2m}{|\beta|}}} \|D^{\alpha-\beta} \tilde{v}\|_{L^{\frac{2m}{|\alpha-\beta|}}} $$
	The Nagliardo-Nirenberg inequality
	$$\|D^j \tilde{u}\|_{L^{\frac{2m}{j}}} \leq C_m \|\tilde{u}\|_{L^\infty}^{1-\frac{j}{m}} \|D^m \tilde{u}\|_{L^2}^{\frac{j}{m}}, \ \ \ \ 0\leq j \leq m,$$
	implies that
	\begin{equation*}
	\begin{aligned}
	\|D^\alpha(\tilde{u}\tilde{v})\|_{L^2}&\leq C_\alpha \sum\limits_{\beta\leq \alpha} C_{\alpha\beta} \|\tilde{u}\|_{L^\infty}^{1-\frac{|\beta|}{m}} \|D^m \tilde{u}\|_{L^2}^{\frac{|\beta|}{m}} \|\tilde{v}\|_{L^\infty}^{1-\frac{|\alpha-\beta|}{m}} \|D^m \tilde{v}\|_{L^2}^{\frac{|\alpha-\beta|}{m}} \\
	&\leq C_\alpha \sum\limits_{\beta\leq \alpha} C_{\alpha\beta} (\|\tilde{u}\|_{L^\infty}\|D^m \tilde{v}\|_{L^2})^{\frac{|\alpha-\beta|}{m}} (\|\tilde{v}\|_{L^\infty}\|D^m \tilde{u}\|_{L^2})^{\frac{|\beta|}{m}} \\
	& \leq C_m(\|\tilde{u}\|_{L^\infty}\|D^m \tilde{v}\|_{L^2}+\|D^m \tilde{u}\|_{L^2}\|\tilde{v}\|_{L^\infty}).
	\end{aligned}
	\end{equation*}
	Summing over all $|\alpha| \leq m$ gives calculus inequality
	$$\|\tilde{u}\tilde{v}\|_{H^m}\leq C_m(\|\tilde{u}\|_{L^\infty}\|D^m \tilde{v}\|_{L^2}+\|D^m \tilde{u}\|_{L^2}\|\tilde{v}\|_{L^\infty}).$$
	So we obtain
	$$ \|uv\|_{H^m}\le c(\|u\|_{L^\infty}\|D^m v\|_{L^2}+\|D^m u\|_{L^2}\|v\|_{L^\infty}).$$
\end{proof}

The following lemma states an interpolation inequality in 1D domain for $H^1-$norm and the anisotropic upper bounds for the integral of the triple product which is similar to  the inequality stated of Lemma 5.1 in \cite{JYWu}.
\begin{lem}\label{lem00}
 Let $\Omega$ is unbounded domain in $\br$, for any $f\in H^1(\Omega)$, then one has
\begin{equation}\label{II2.1}
\|f\|_{L^\infty}\leqslant C\|f\|^{\frac{1}{2}}_{L^2}\| f'\|^{\frac{1}{2}}_{L^2}.
\end{equation}
Moreover,for any $f, g,h \in L^{2}(\br^3_+)$ with $\partial_{1}f, \partial_{2}g, \partial_{3}h\in L^{2}(\br^3_+)$. Then
\begin{equation}\label{II2.2}
\begin{aligned}
\left|\int_{\br^3_+} f g h \dd x\right| \leqslant C\|f\|_{L^{2}}^{\frac{1}{2}}\|\partial_{1} f\|_{L^{2}}^{\frac{1}{2}}\|g\|_{L^{2}}^{\frac{1}{2}}\|\partial_{2}
g\|_{L^{2}}^{\frac{1}{2}}\|h\|_{L^{2}}^{\frac{1}{2}}\|\partial_{3} h\|_{L^{2}}^{\frac{1}{2}}.
\end{aligned}
\end{equation}
\end{lem}
\begin{proof} As $\Omega$ is an unbounded domain in $\br$, without loss of generality, let $\Omega=\br_+.$  For any $f\in H^1(\br_+)$, it is easy to check that $\lim\limits_{x\to+\infty} f(x)=0.$ Thus we know for any $x,A\in \br_+$
\begin{equation*}
\begin{aligned}
|f(x)|^2-|f(A)|^2&=2\int_A^xf'(t)f(t)\dd t\\
&\leqslant2 \int_0^\infty|f'(x)||f(x)|\dd x\leqslant2\|f\|_{L^2}\| f'\|_{L^2}
\end{aligned}
\end{equation*}
Let $A\to\infty$, we can conclude that the inequality \eqref{II2.1} holds.
To prove  the inequality \eqref{II2.2}, by the inequality \eqref{II2.1} frequently, we can calculate
\begin{equation*}
\begin{aligned}
\left|\int_{\br^3_+} f g h \dd x\right| &\leqslant \int_{\br^2}|fg|_{L^1_{x_3}} \|h\|_{L^\infty_{x_3}}\dd x_h\leqslant \|h\|_{L^\infty_{x_3}L^2_{x_h}}\left(\int_{\br^2}\|fg\|^2_{L^1_{x_3}}  \dd x_h\right)^\frac{1}{2}\\
&\leqslant C\|h\|_{L^{2}}^{\frac{1}{2}}\|\partial_{3} h\|_{L^{2}}^{\frac{1}{2}}\left(\int_{\br}\int_{\br}\|f\|_{L^2_{x_3}}^2\|g\|^2_{L^2_{x_3}}  \dd x_1\dd x_2\right)^\frac{1}{2}\\
&\leqslant   C\|h\|_{L^{2}}^{\frac{1}{2}}\|\partial_{3} h\|_{L^{2}}^{\frac{1}{2}} \|f\|_{L^\infty_{x_1} L^2_{x_3,x_2}}\|g\| _{L^\infty_{x_2} L^2_{x_3,x_1} }\\
&\leqslant  C\|f\|_{L^{2}}^{\frac{1}{2}}\|\partial_{1} f\|_{L^{2}}^{\frac{1}{2}}\|g\|_{L^{2}}^{\frac{1}{2}}\|\partial_{2}
g\|_{L^{2}}^{\frac{1}{2}}\|h\|_{L^{2}}^{\frac{1}{2}}\|\partial_{3} h\|_{L^{2}}^{\frac{1}{2}}.
\end{aligned}
\end{equation*}

\end{proof}

A necessary step in the proof of Theorem \ref{T1.1} is the following global $H^2$
-bound.
\begin{lem}\label{lem0}. Assume that initial data $(u_0,\theta_0)\in H^2(\br_+^3)$ with boundary conditions \eqref{N1.5} and \eqref{N1.6} , satisfies $\nabla\cdot u_0=0.$ Then there exists $\varepsilon>0$ such that, if
	$$\|u_0\|_{H^2}+\|\theta_0\|_{H^2}\leqslant\varepsilon,$$
	then the system \eqref{E1.9} with boundary conditions \eqref{N1.5} and \eqref{N1.6}  has a unique global solution $(u,\theta)\in L^\infty(0,\infty; H^2(\br_+^3))$ satisfying, for a constant $C>0$ and for all $t>0,$
	\begin{equation}\label{leI2.3}
		\begin{aligned}
		\|u\|^2_{H^2}+\|\theta\|^2_{H^2}+\int_0^t\left(\|\nabla_h u(\tau)\|^2_{H^2}+\|\nabla_h\theta(\tau)\|^2_{H^2}\right)\dd\tau\leqslant C\varepsilon^2.
		\end{aligned}
	\end{equation}
\end{lem}

\begin{proof} Since the local (in time) well-posedness of (1.3) can be established via a standard approach (see \cite{wx3}), our attention is focused on the global bound of $(u, \theta)$. The framework of the proof is the bootstrapping argument. Define the energy functional $\tilde{E}(t)$ by
$$\tilde{E}(t) = \mathop {\sup }\limits_{0 \leqslant \tau  \leqslant t} \{ {\| {u(\tau )} \|_{{H^2}}^2 + \| {\theta (\tau )} \|_{{H^2}}^2} \} + 2 \int_0^t {\| {{\nabla _h}u(\tau )} \|_{{H^2}}^2d\tau }  + 2\int_0^t {\| {\nabla _h}\theta (\tau )\|_{{H^2}}^2d\tau } .$$
Our main efforts are devoted to showing that, for a constant $C>0$ and for $t>0$,
\begin{equation}\label{w-1}
\tilde{E}(t) \leqslant \tilde{E}(0) + C{\tilde{E}^{{\textstyle{3 \over 2}}}}(t).
\end{equation}
Once (\ref{w-1}) is shown, then a direct application of the bootstrapping argument implies that, if
\begin{equation}\label{w-2}
\tilde{E}(0) = \| {({u_0},{\theta _0})}\|_{{H^2}}^2 \leqslant \varepsilon^2 : = \frac{1}{{16{C^2}}},
\end{equation}
then
\begin{equation}\label{w-3}
\tilde{E}(t) \leqslant \frac{1}{{8{C^2}}} \quad t>0.
\end{equation}
In fact, if we make the ansatz that
\begin{equation}\label{w-4}
\tilde{E}(t) \leqslant \frac{1}{{4{C^2}}}.
\end{equation}
Inserting (\ref{w-4}) in (\ref{w-1}) and invoking (\ref{w-2}) yields
$$\tilde{E}(t) \leqslant \tilde{E}(0) + \frac{1}{2}\tilde{E}(t),$$
that is
$$\tilde{E}(t) \leqslant 2\tilde{E}(0) \leqslant \frac{1}{{8{C^2}}},$$
which is only half of the bound in the ansatz in (\ref{w-4}). Then the bootstrapping argument implies (\ref{w-3}).
Next, we prove the energy inequality (\ref{w-1}). Due to the properties of the singular inequalities and the interpolations, one have got  the equivalence of the norms
$$\|(u,\theta)\|_{H^{2}}^{2}\sim \|(u,\theta)\|_{L^{2}}^{2}+\|(\Delta u,\Delta\theta)\|_{L^{2}}^{2},$$
here  the norm $\|(u,\theta)\|_{X}=\|u\|_{X}+\|\theta\|_{X}.$
It suffices to bound $L^2$ and the homogeneous $\dot{H}^2$-norm of $(u,\theta)$. By a simple energy estimate and $\nabla  \cdot u = 0$, we obtain the global $L^2$-norm of $(u,\theta)$ obeys
\begin{equation}\label{w-5}
\begin{aligned}
\frac{1}{2} \frac{d}{d t}\|(u,\theta)\|_{L^{2}}^{2}+\|\nabla _{h}u \|_{L^{2}}^{2}+\|\nabla _{h}\theta\|_{L^{2}}^{2}=0.
\end{aligned}
\end{equation}
To obtain the homogeneous $\dot{H}^2$ norm of $(u,\theta)$,  we apply $\Delta$ operator to the equations of \eqref{E1.9} and take the $L^2-$
inner product of the resulting equations with $(\Delta u,\Delta\theta)$. Then integrating by parts and by the boundary conditions \eqref{N1.5}, \eqref{N1.6}, we obtain
\begin{equation}\label{w-6}
\begin{aligned}
\frac{1}{2}\frac{d}{d t} \|(\Delta u,\Delta\theta)\|_{L^{2}}^{2}+\|\nabla _{h} \Delta u
\|_{L^{2}}^{2}+\|\nabla _{h} \Delta\theta\|_{L^{2}}^{2}=I_1+I_2,
\end{aligned}
\end{equation}
where $I_1$ and $I_2$ are given by
\begin{align*}
&I_1=- \int_{\br^3_{+}} \Delta(u \cdot \nabla u) \cdot \Delta u  d x, \\
&I_2=- \int_{\br^{3}_{+}} \Delta(u \cdot \nabla \theta) \Delta \theta dx.
\end{align*}
Here we used the fact that for $\Delta u_3=\Delta_h u_3+ \partial_3^2 u_3=0$ on $x_3=0$.
To estimate $I_1$, we decompose it as
\begin{equation*}
\begin{aligned}
{I_1} =&- \sum\limits_{i = 1}^2 {\int_{\br^{3}_{+}} {\partial _i^2(u \cdot \nabla u) \cdot \Delta u} } dx - \sum\limits_{k= 1}^2 {\int_{\br^{3}_{+}} {\partial _3^2(u \cdot \nabla u) \cdot \partial _k^2 u} } dx-\int_{\br^{3}_{+}} {\partial _3^2(u \cdot \nabla u) \cdot \partial_3^2 u} dx\\
=&I_{11} + I_{12}+ I_{13}.
\end{aligned}
\end{equation*}
Due to the boundary conditions \eqref{N1.5}, \eqref{N1.6} and $\nabla \cdot u =0$, we have
\begin{equation*}
\begin{aligned}
I_{11} +I_{12}= & - \sum\limits_{i = 1}^2\sum_{k=1}^{2} {\int_{\br^{3}_{+}} {\partial _i^2 \nabla \cdot (u \otimes u) \cdot \partial _k^2 u} } dx + \sum\limits_{i = 1}^2 {\int_{\mathbb{R}^{3}_{+}} {\partial _i \nabla \cdot (u \otimes u) \cdot \partial_i \partial _3^2 u} } dx
- \sum\limits_{k= 1}^2 {\int_{\br^{3}_{+}} {\partial _3^2\partial _k (u \otimes u) \cdot \nabla \partial _k u} } dx \\
\leqslant& C \sum\limits_{i = 1}^2\sum_{k=1}^{2} \| \partial _i  (u \otimes u)\|_{H^2} \|\partial _k^2 u\|_{L^2}+\sum\limits_{i = 1}^2 \| \partial _i  (u \otimes u)\|_{H^1} \|\partial_i \partial _3^2 u\|_{L^2}+\sum_{k=1}^{2} \|\partial _k (u \otimes u)\|_{H^2} \|\nabla \partial _k  u\|_{L^2}\\
\leqslant& C \|u\|_{L^\infty}\| {\nabla _h} u\|_{H^2}^2
\leqslant C \|u\|_{H^2}\| {\nabla _h} u\|_{H^2}^2.
\end{aligned}
\end{equation*}
To estimate $I_3$, we decompose it as
\begin{equation*}
\begin{aligned}
I_{13}= -\sum\limits_{j = 1}^2 \int_{\br^{3}_{+}} {\partial _3^2(u_j \partial_j u) \cdot \partial_3^2 u} dx
- \int_{\br^{3}_{+}} {\partial _3^2(u_3 \partial_3 u) \cdot \partial_3^2 u} dx
=I_{131}+I_{132}
\end{aligned}
\end{equation*}
By the Leibniz Formula and
\begin{equation*}
\begin{aligned}
\int_{\br^{3}_{+}} {(u \cdot \nabla\partial _3^2 u) \cdot \partial _3^2 u}  dx
=-\int_{\br^{3}_{+}} {(\nabla \cdot u) |\partial _3^2 u|^2}  dx +\int_{x_3=0} {(u \cdot n) |\partial _3^2 u|^2}  dx=0,
\end{aligned}
\end{equation*}
we have
\begin{equation*}
\begin{aligned}
 I_{131}=&- \sum\limits_{j = 1}^2 \sum\limits_{l = 1}^2 {\mathcal{C}_2^l\int_{\br^3_+}{\partial _3^l {u_j}\partial_3^{2 - l}{\partial _j}u \cdot \partial _3^2 u} dx} \\
\leqslant & C\sum\limits_{j = 1}^2 \sum\limits_{l = 1}^2 \| \partial _3^l u_j\|_{L^2}^{\frac{1}{2}} \| {\partial _1}\partial _3^l u_j\|_{L^2}^{\frac{1}{2}} \| {\partial _3^{2 - l}{\partial _j}u} \|_{L^2}^{\frac{1}{2}} \| \partial _3 \partial _3^{2 - l}{\partial _j}u\|_{L^2}^{\frac{1}{2}} \| {\partial _3^2} u \|_{L^2}^{\frac{1}{2}} \| {\partial _2}\partial _3^2 u\|_{L^2}^{\frac{1}{2}} \\
\leqslant & C\| u \|_{H^2}\| {\nabla _h}u\|_{H^2}^2.
\end{aligned}
\end{equation*}
By $\nabla  \cdot u = 0$ and Lemma \ref{lem00},
\begin{equation*}
\begin{aligned}
I_{132} = & - \sum\limits_{l = 1}^2 {\mathcal{C}_2^l\int_{\br^{3}_{+}}  {\partial _3^l{u_3}\partial _3^{2 - l}{\partial _3}u \cdot \partial _3^2 u} dx} \\
\leqslant & \sum\limits_{l = 1}^2 {\mathcal{C}_2^l\int_{\br^{3}_{+}}  {\partial _3^{l - 1}({\nabla _h} \cdot {u_h})\partial _3 \partial _3^{2 - l}u \cdot \partial _3^2 u} dx} \\
\leqslant & C\sum\limits_{l = 1}^2 \| \partial _3^{l - 1}{\nabla _h} \cdot {u_h} \|_{L^2}^{\frac{1}{2}} \| {\partial _3}\partial _3^{l - 1}{\nabla _h} \cdot {u_h}\|_{L^2}^{\frac{1}{2}} \| \partial _3^{2 - l} \partial _3u\|_{L^2}^{\frac{1}{2}} \| {\partial _1}\partial _3^{2 - l}\partial _3 u \|_{L^2}^{\frac{1}{2}} \| \partial _3^2 u \|_{L^2}^{\frac{1}{2}} \| {\partial _2}\partial _3^2 u\|_{L^2}^{\frac{1}{2}} \\
\leqslant & C\| u \|_{H^2}\| {\nabla _h}u \|_{H^2}^2.
\end{aligned}
\end{equation*}
Therefore,
\begin{equation}\label{w-8}
I_{1}\leqslant C\| u\|_{H^{2} }\|\nabla _{h} u\|_{H^{2} }^{2}.
\end{equation}
Now we turn to estimate $I_2$. Due to $\nabla \cdot u =0$, We further decompose it as
\begin{equation*}
\begin{aligned}
{I_2} = &- \sum\limits_{i = 1}^2 {\int_{\br^{3}_{+}}  {\partial _i^2\nabla\cdot  (u \theta )\Delta\theta } } dx- \sum\limits_{k= 1}^2 {\int_{\br^{3}_{+}}  {\partial _3^2\nabla\cdot  (u \theta )\partial _k^2\theta } } dx- \int_{\br^{3}_{+}}  {\partial _3^2(u\cdot \nabla \theta )\partial _3^2\theta } dx \\
=& I_{21}+ I_{22}+I_{23}
\end{aligned}
\end{equation*}
$I_{21}$ and $I_{22}$ can be estimated similarly as $I_{11}$ and $I_{12}$.  In fact,
\begin{equation*}
\begin{aligned}
I_{21}+I_{22}
\leqslant C (\|u\|_{H^2}+\|\theta\|_{H^2})(\| {\nabla _h} u\|_{H^2}^2+\| {\nabla _h} \theta\|_{H^2}^2).
\end{aligned}
\end{equation*}
Similarly, by Young's inequality, Sobolev's inequality and Lemma \ref{lem00}, we have
\begin{equation*}
\begin{aligned}
I_{23}
=& - \sum\limits_{j = 1}^2 {\sum\limits_{l = 1}^2 {\mathcal{C}_2^l\int {\partial _3^l {u_j}\partial_3^{2 - l}{\partial _j}\theta\cdot \partial _3^2 \theta} dx} }
- \sum\limits_{l = 1}^2 {\mathcal{C}_2^l\int {\partial _3^l{u_3}\partial _3^{2 - l}{\partial _3}\theta \cdot \partial _3^2 \theta} dx} \\
\leqslant & C\sum\limits_{j = 1}^2 \sum\limits_{l = 1}^2 \| \partial _3^l u_j\|_{L^2}^{\frac{1}{2}} \| {\partial _1}\partial _3^l u_j\|_{L^2}^{\frac{1}{2}} \| {\partial _3^{2 - l}{\partial _j}\theta} \|_{L^2}^{\frac{1}{2}} \| \partial _3 \partial _3^{2 - l}{\partial _j}\theta\|_{L^2}^{\frac{1}{2}} \| {\partial _3^2} \theta \|_{L^2}^{\frac{1}{2}} \| {\partial _2}\partial _3^2 \theta\|_{L^2}^{\frac{1}{2}} \\
&+C\sum\limits_{l = 1}^2 \| \partial _3^{l - 1}{\nabla _h} \cdot {u_h} \|_{L^2}^{\frac{1}{2}} \| {\partial _3}\partial _3^{l - 1}{\nabla _h} \cdot {u_h}\|_{L^2}^{\frac{1}{2}}
\| \partial _3^{2 - l} \partial _3\theta\|_{L^2}^{\frac{1}{2}} \| {\partial _1}\partial _3^{2 - l}\partial _3 \theta \|_{L^2}^{\frac{1}{2}} \| \partial _3^2 \theta\|_{L^2}^{\frac{1}{2}} \| {\partial _2}\partial _3^2 \theta\|_{L^2}^{\frac{1}{2}} \\
\leqslant & C(\| u \|_{H^2} ^{\frac{1}{2}} \| {\nabla _h}u \|_{H^2} ^{\frac{1}{2}} \| {\nabla _h}\theta\|_{H^2} ^{\frac{3}{2}} \| \theta\|_{H^2} ^{\frac{1}{2}}
+\| {\nabla _h}u \|_{H^2}\| \theta\|_{H^2} \| {\nabla _h}\theta\|_{H^2} )\\
\leqslant & C(\| u \|_{H^2} + \| \theta  \|_{H^2})(\| {\nabla _h}u \|_{H^2}^2 + \| {\nabla _h}\theta  \|_{H^2}^2).
\end{aligned}
\end{equation*}
Collecting the bound for $I_2$, we obtain
\begin{equation}\label{w-9}
I_{2}\leqslant C(\| u\|_{H^{2} }+\|\theta\|_{H^2})(\|\nabla _{h} u\|_{H^{2} }^{2}+\| {\nabla _h}\theta\|_{H^2}^2).
\end{equation}
Inserting (\ref{w-8}) and (\ref{w-9}) in (\ref{w-6}), integrating in time over $[0,t]$ and adding \eqref{w-5}, we deduce
\begin{equation*}
\begin{aligned}
\tilde{E}(t)&\leqslant \tilde{E}(0)+C\int_0^t \|u\|_{H^2}\|\nabla_hu\|_{H^2}^2 d\tau+C\int_0^t(\|u\|_{H^2}+\|\theta\|_{H^2})(\| {{\nabla _h}u} \|_{{H^2}}^2 + \|  {\nabla _h}\theta  \|_{{H^2}}^2) d\tau\\
&\leqslant \tilde{E}(0)+C\tilde{E}^{\frac{3}{2}}(t).\\
\end{aligned}
\end{equation*}
which is the desired inequality (\ref{w-1}).

It's easy to prove the uniqueness result of Lemma \ref{lem0}. Let $(u^{(1)}, p^{(1)}, \theta^{(1)})$ and $(u^{(2)}, p^{(2)}, \theta^{(2)})$ be two solutions of equation (\ref{E1.9}) with one of them in the regularity
class, say $(u^{(1)},\theta^{(1)})\in L^{\infty}(0,\infty;H^{2}(\br^{3}_+))$ must coincide. In fact, their difference $(\tilde{u}, \tilde{p}, \tilde{\theta})$ with
$$\tilde{u}=u^{(2)}-u^{(1)}, \quad \tilde{p}=p^{(2)}-p^{(1)}, \quad \tilde{\theta}=\theta^{(2)}-\theta^{(1)}$$
satisfies
\begin{eqnarray}
\begin{cases}
\partial_{t} \tilde{u}+u^{(2)} \cdot \nabla \tilde{u}+\tilde{u} \cdot \nabla u^{(1)}=-\nabla \tilde{p}+ \Delta_{h} \tilde{u}+\tilde{\theta} e_{3}, &x\in\br^3_+,t>0\\
\partial_{t} \tilde{\theta}+u^{(2)} \cdot \nabla \tilde{\theta}+\tilde{u} \cdot \nabla \theta^{(1)}+\tilde{u}_{3}-\Delta_{h} \tilde{\theta}=0, &x\in\br^3_+,t>0\\
\nabla \cdot \tilde{u}=0, &x\in\br^3_+,t>0\\
\tilde{u}(x, 0)=0, \tilde{\theta}(x, 0)=0,&x\in\br^3_+\\
\frac{\partial \tilde{u}_1}{\partial x_3}=\frac{\partial \tilde{u}_2}{\partial x_3}=0, \tilde{u}_3=0,\tilde{\theta}=0,&~\mbox{on}~ x_3=0,t>0
\end{cases}\label{w-10}
\end{eqnarray}
Taking the $L^{2}$-inner product of (\ref{w-10}) with $(\tilde{u},\tilde{\theta})$, by Lemma \ref{lem00}, Young's inequality and the global bounds for
$\|(u^{(1)},\theta^{(1)})\|_{H^{2}}$, we deduce
\begin{equation*}
\begin{aligned}
&\frac{1}{2} \frac{d}{d t}\|(\tilde{u},\tilde{\theta})\|_{L^{2}}^{2}+\left\|\nabla _{h} \tilde{u} \right\|_{L^{2}}^{2}+\|\nabla _{h} \tilde{\theta}\|_{L^{2}}^{2}\\
=&-\int_{\br^{3}_{+}} (\tilde{u} \cdot \nabla u^{(1)} ) \cdot \tilde{u}d x-\int_{\br^{3}_{+}} (\tilde{u} \cdot \nabla \theta^{(1)} ) \cdot \tilde{\theta} d x\\
\leqslant& C\|\tilde{u}\|_{L^{2} }^{\frac{1}{2}}\|\partial_{1}\tilde{u}\|_{L^{2} }^{\frac{1}{2}}\|\nabla u^{(1)}\|_{L^{2} }^{\frac{1}{2}}\|\partial_{3}\nabla u^{(1)}\|_{L^{2}
}^{\frac{1}{2}}\|\tilde{u}\|_{L^{2} }^{\frac{1}{2}}\|\partial_{2}\tilde{u}\|_{L^{2} }^{\frac{1}{2}}\\
&+C\|\tilde{u}\|_{L^{2} }^{\frac{1}{2}}\|\partial_{1}\tilde{u}\|_{L^{2} }^{\frac{1}{2}}\|\nabla \theta^{(1)}\|_{L^{2} }^{\frac{1}{2}}\|\partial_{3}\nabla \theta^{(1)}\|_{L^{2}
}^{\frac{1}{2}}\|\tilde{\theta}\|_{L^{2} }^{\frac{1}{2}}\|\partial_{2}\tilde{\theta}\|_{L^{2} }^{\frac{1}{2}}\\
\leqslant& C\|\tilde{u}\|_{L^{2}}\|\nabla_{h}\tilde{u}\|_{L^{2} }+C\|\tilde{u}\|_{L^{2}}^{\frac{1}{2}}\|\nabla_{h} \tilde{u}\|_{L^{2}}^{\frac{1}{2}}\|\tilde{\theta}\|_{L^{2}}^{\frac{1}{2}}\|\nabla_{h} \tilde{\theta}\|_{L^{2}}^{\frac{1}{2}}\\
\leqslant& \frac{1}{2}\|\nabla_{h} \tilde{u}\|_{L^{2} }^{2}+\frac{1}{2}\|\nabla_{h} \tilde{\theta}\|_{L^{2} }^{2}+C\|(\tilde{u},\tilde{\theta})\|_{L^{2}}^{2}.
\end{aligned}
\end{equation*}
where we have used the fact that
$$\int\tilde{\theta}e_{3} \cdot \tilde{u}dx -\int \tilde{u}_{3}\tilde{\theta}dx=0. $$
Then we apply the Gr\"{o}nwall's inequality to get the desired global uniqueness,
$$\|\tilde{u}\|_{L^{2}}^{2}=\|\tilde{\theta}\|_{L^{2}}^{2}=0.$$
Thus the proof of Lemma \ref{lem0} is completed.
\end{proof}

According to the boundary conditions \eqref{N1.5}, \eqref{N1.6} and Lemma \ref{lem0}, we observe that the following results which are critical to finish the proof of Theorem \ref{T1.1}.
\begin{lem}\label{lem01}
	Assume that initial data $u_0$ and $\theta_0$ with boundary conditions \eqref{N1.5} and \eqref{N1.6}  and $\p_3^2\theta_0\left|_{x_3=0}\right.=0,$ then if the triple $(u,p,\theta)$ is the smooth solution of the problem \eqref{E1.9} with $\int_0^t\|\nabla_h u_h\|_{L^\infty(\br_+^3)}<\infty$, we find that
	$$\p_3^3 p(x_h,0,t)=0, ~\mbox{for}~x_h\in\br^2, t>0.$$
\end{lem}
\begin{proof}
	Since \eqref{N1.5}, we know that $\p_3^2u_3\left|_{x_3=0}\right.=\p_3(\nabla_h\cdot u_h)\left|_{x_3=0}\right.=\nabla_h\cdot\p_3 u_h\left|_{x_3=0}\right.=0.$
	
	Applying the operator $\p_3^2$ to the third  equation of \eqref{E1.9} , we have got
	\begin{equation*}
	\p_t(\p_3^2u_3)-\Delta_h\p_3^2 u_3+\p_3^2(u\cdot\nabla u_3)+\p_3^3p=\p_3^2\theta.
	\end{equation*}
	Since
	\begin{equation*}
	\begin{aligned}
	\p_3^2(u\cdot\nabla u_3)&=\p_3^2u\cdot\nabla u_3+2\p_3u\cdot\nabla \p_3u_3+u\cdot\nabla\p_3^2u_3\\
	&=\p_3^2u_h\cdot\nabla_h u_3+\p_3^2 u_3\p_3 u_3\\
	&+2\p_3 u_h\cdot\nabla_h\p_3u_3+2\p_3 u_3\p^2_3 u_3\\
	&+u_h\cdot\nabla_h\p_3^2 u_3+u_3\p_3^3u_3.
	\end{aligned}
	\end{equation*}
	By \eqref{N1.5} and $\p_3^2u_3=0$, on $x_3=0$,  we have $\p_3^2(u\cdot\nabla u_3)=0$, on $x_3=0$. Thus $\p_3^3p=\p_3^2\theta$,  on $x_3=0$.
	Applying the operator $\p_3^2$ to the fourth  equation of \eqref{E1.9} , we obtain
	\begin{equation}\label{E3.2}
	\p_t(\p_3^2\theta)-\Delta_h\p_3^2 \theta+\p_3^2(u\cdot\nabla \theta)+\p_3^2u_3=0.
	\end{equation}
	By the boundary conditions \eqref{N1.5}, \eqref{N1.6} and $\p_3^2u_3=0$  on $x_3=0,$  one refers that
	\begin{equation}\label{E3.3}
	\frac{\dd}{\dd t}\|\p_3^2\theta\|_{L^2(\br^2)}^2+\|\nabla_h\p_3^2 \theta\|_{L^2(\br^2)}^2\leqslant\|\nabla_h u_h(\cdot,0)\|_{L^\infty(\br^2)}\|\p_3^2\theta\|_{L^2(\br^2)}^2\leqslant\|\nabla_h u_h \|_{L^\infty(\br_+^3)}\|\p_3^2\theta\|_{L^2(\br^2)}^2.
	\end{equation}
	By Gronwall's inequality, we obtain that
	$$\|\p_3^2\theta(\cdot,0)\|_{L^2(\br^2)}^2\leqslant \|\p_3^2\theta_0(\cdot,0)\|_{L^2(\br^2)}^2 e^{\int_0^t\|\nabla_h u_h \|_{L^\infty }\dd \tau}.$$
	
	When $\|\p_3^2\theta_0(\cdot,0)\|_{L^2(\br^2)}=0,$ then  $\|\p_3^2\theta(\cdot,0, t)\|_{L^2(\br^2)}^2=0.$ Therefore, $\p_3^3p\left|_{x_3=0}\right.=0, $ for all $t>0.$
\end{proof}

\textbf{Proof of Theorem \ref{T1.1}}
We prove the existence of the solution by using the bootstrapping argument. Define the energy functional $E(t)$ by
$$E(t) = \mathop {\sup }\limits_{0 \leqslant \tau  \leqslant t} \{ {\| {u(\tau )} \|_{{H^3}}^2 + \| {\theta (\tau )} \|_{{H^3}}^2} \} + 2\int_0^t {\| {{\nabla _h}u(\tau )} \|_{{H^3}}^2d\tau }  + 2 \int_0^t {\| {\nabla _h}\theta (\tau )\|_{{H^3}}^2d\tau } .$$
Our main efforts are devoted to showing that, for a constant $C_0>0$ and for $t>0$,
\begin{equation}\label{w-0}
E(t) \leqslant E(0) + {C_0}{E^{{\textstyle{3 \over 2}}}}(t).
\end{equation}
Due to the Lemma \ref{lem00}, we focus on the $\dot{H}^3$ norm of $(u, \theta)$. Applying $\partial _i^3(i = 1,2,3)$  to (\ref{E1.9}), dotting by $(\partial _i^3 u,\partial _i^3\theta )$ and  integrating by parts, we have
\begin{equation}\label{w-11}
\begin{aligned}
&\frac{1}{2}\frac{d}{d t} \sum_{i=1}^{3}\|(\partial_{i} ^{3}u,\partial_{i} ^{3}\theta)\|_{L^{2}}^{2}+\sum_{i=1}^{3}\|\nabla _{h} \partial_{i} ^{3} u
\|_{L^{2}}^{2}+ \sum_{i=1}^{3}\|\nabla _{h} \partial_{i} ^{3}\theta\|_{L^{2}}^{2}+\int_{x_3=0}\p_i^3 p\p^3_i u_3\dd\sigma=K_1+K_2,
\end{aligned}
\end{equation}
where $K_1$ and $K_2$ are given by
\begin{align*}
&K_1=- \sum_{i=1}^{3}\int_{\br^{3}_{+}} \partial_{i} ^{3}(u \cdot \nabla u) \cdot \partial_{i} ^{3} u  d x, \\
&K_2=- \sum_{i=1}^{3}\int_{\br^{3}_{+}} \partial_{i} ^{3}(u \cdot \nabla \theta) \partial_{i} ^{3} \theta dx.
\end{align*}
Now we must dealt with the boundary integration $S=\int_{x_3=0}\p_i^3 p\p^3_i u_3\dd\sigma$. Clearly, when $i=1,2$, then $S=0$ by \eqref{N1.5}.  By the estimates \eqref{leI2.3},
 $$\int_0^t\|\nabla_h u_h \|_{L^\infty}\dd\tau\leqslant\left (\int_0^t\|\nabla_h u_h \|^2_{L^\infty}\dd\tau\right)^\frac{1}{2}t^\frac{1}{2}\leqslant C\left(\int_0^t\|\nabla_h u_h \|^2_{H^2}\dd\tau\right)^\frac{1}{2}t^\frac{1}{2}<\infty $$
 for any fixed $t>0.$ Thus for $i=3$, from Lemma \ref{lem01}, we know that
\begin{equation*}
S=\int_{x_3=0}\p_3^3 p\p^3_3 u_3\dd\sigma=0.
\end{equation*}

To estimate $K_1$, we decompose it as
\begin{equation*}
\begin{aligned}
K_1 =&- \sum\limits_{i = 1}^2 {\int_{\br^{3}_{+}} {\partial _i^3(u \cdot \nabla u) \cdot \partial _i^3 u} } dx - \sum\limits_{j = 1}^2 {\int_{\br^{3}_{+}} {\partial _3^3 ({u_j}{\partial _j}u) \cdot \partial _3^ 3 u} } dx - \int_{\br^{3}_{+}} {\partial _3^3({u_3} {\partial _3}u) \cdot \partial _3^3 u} dx\\
=&K_{11}+ K_{12} + K_{13}.
\end{aligned}
\end{equation*}
$K_{11}$ is easy to bound. Due to $\nabla \cdot u =0$ and Lemma \ref{lem000},
\begin{equation*}
\begin{aligned}
K_{11} =  - \sum\limits_{i = 1}^2 {\int_{\br^{3}_{+}} {\partial _i^3 \nabla \cdot (u \otimes u) \cdot \partial _i^3 u} } dx
\leqslant C \|u\|_{L^\infty}\| {\nabla _h} u\|_{H^3}^2
\leqslant C \|u\|_{H^3}\| {\nabla _h} u\|_{H^3}^2.
\end{aligned}
\end{equation*}
By the Leibniz formula and
\begin{equation*}
\begin{aligned}
\int_{\br^{3}_{+}} {(u \cdot \nabla\partial _i^3 u) \cdot \partial _i^3 u}  dx
=-\int_{\br^{3}_{+}} {(\nabla \cdot u) |\partial _i^3 u|^2}  dx +\int_{x_3=0} {(u \cdot n) |\partial _i^3 u|^2}  dx=0, \ \ \ \ i=1,2,3,
\end{aligned}
\end{equation*}
we have
\begin{equation*}
\begin{aligned}
K_{12} = & - \sum\limits_{j = 1}^2 {\sum\limits_{l = 1}^3 {\mathcal{C}_3^l\int {\partial _3^l {u_j}\partial_3^{3 - l}{\partial _j}u \cdot \partial _3^3 u} dx} } \\
\leqslant & C\sum\limits_{j = 1}^2 \sum\limits_{l = 1}^3 \| \partial _3^l u_j\|_{L^2}^{\frac{1}{2}} \| {\partial _1}\partial _3^l u_j\|_{L^2}^{\frac{1}{2}} \| {\partial _3^{3 - l}{\partial _j}u} \|_{L^2}^{\frac{1}{2}} \| \partial _3 \partial _3^{3 - l}{\partial _j}u\|_{L^2}^{\frac{1}{2}} \| {\partial _3^2} u \|_{L^2}^{\frac{1}{2}} \| {\partial _2}\partial _3^3 u\|_{L^2}^{\frac{1}{2}} \\
\leqslant & C\| u \|_{H^3}\| {\nabla _h}u\|_{H^3}^2.
\end{aligned}
\end{equation*}
By $\nabla  \cdot u = 0$ and Lemma \ref{lem0},
\begin{equation*}
\begin{aligned}
K_{13}= & - \sum\limits_{l = 1}^3 {\mathcal{C}_3^l\int {\partial _3^l{u_3}\partial _3^{3 - l}{\partial _3}u \cdot \partial _3^3 u} dx} \\
\leqslant & \sum\limits_{l = 1}^3 {\mathcal{C}_3^l\int {\partial _3^{l - 1}({\nabla _h} \cdot {u_h})\partial _3^{3 - l} \partial _3 u \cdot \partial _3^3 u} dx} \\
\leqslant & C\sum\limits_{l = 1}^3 \| \partial _3^{l - 1}{\nabla _h} \cdot {u_h} \|_{L^2}^{\frac{1}{2}} \| {\partial _3}\partial _3^{l - 1}{\nabla _h} \cdot {u_h}\|_{L^2}^{\frac{1}{2}} \| \partial _3^{3 - l} \partial _3u\|_{L^2}^{\frac{1}{2}} \| {\partial _1}\partial _3^{3 - l}\partial _3 u \|_{L^2}^{\frac{1}{2}} \| \partial _3^2 u \|_{L^2}^{\frac{1}{2}} \| {\partial _2}\partial _3^3 u\|_{L^2}^{\frac{1}{2}} \\
\leqslant & C\| u \|_{H^3}\| {\nabla _h}u \|_{H^3}^2.
\end{aligned}
\end{equation*}
Therefore,
\begin{equation}\label{w-12}
K_{1}\leqslant C\| u\|_{H^{3} }\|\nabla _{h} u\|_{H^{3} }^{2}.
\end{equation}
Now we turn to estimate $K_2$. Due to $\nabla \cdot u =0$, We further decompose it as
\begin{equation*}
\begin{aligned}
K_2= &- \sum\limits_{i = 1}^2 {\int {\partial _i^3\nabla\cdot  (u \theta )\partial _i^3\theta } } dx - \int {\partial _3^3(u\cdot \nabla \theta )\partial _3^3\theta } dx \\
=& K_{21}+ K_{22}.
\end{aligned}
\end{equation*}
To deal with $K_{21}$, we apply Sobolev's inequality
\begin{equation*}
\begin{aligned}
K_{21}
\leqslant& C \sum\limits_{i = 1}^2 (\| \partial _i^3 \nabla u\|_{L^2}\|\theta\|_{L^\infty} +\| \partial _i^3 \nabla \theta\|_{L^2}\|u\|_{L^\infty})
\|\partial _i^3 \theta\|_{L^2}\\
\leqslant& C (\|u\|_{H^3}+\|\theta\|_{H^3})(\| {\nabla _h} u\|_{H^3}^2+\| {\nabla _h} \theta\|_{H^3}^2).
\end{aligned}
\end{equation*}
By Young's inequality, Sobolev's inequality and Lemma \ref{lem00}, we have
\begin{equation*}
\begin{aligned}
K_{22}
=& - \sum\limits_{j = 1}^2 {\sum\limits_{l = 1}^3 {\mathcal{C}_3^l\int {\partial _3^l {u_j}\partial_3^{3 - l}{\partial _j}\theta\cdot \partial _3^3 \theta} dx} }
- \sum\limits_{l = 1}^3 {\mathcal{C}_3^l\int {\partial _3^l{u_3}\partial _3^{3- l}{\partial _3}\theta \cdot \partial _3^3 \theta} dx} \\
\leqslant & C\sum\limits_{j = 1}^2 \sum\limits_{l = 1}^3 \| \partial _3^l u_j\|_{L^2}^{\frac{1}{2}} \| {\partial _1}\partial _3^l u_j\|_{L^2}^{\frac{1}{2}} \| {\partial _3^{3 - l}{\partial _j}\theta} \|_{L^2}^{\frac{1}{2}} \| \partial _3 \partial _3^{3 - l}{\partial _j}\theta\|_{L^2}^{\frac{1}{2}} \| {\partial _3^3} \theta \|_{L^2}^{\frac{1}{2}} \| {\partial _2}\partial _3^3 \theta\|_{L^2}^{\frac{1}{2}} \\
&+C\sum\limits_{l = 1}^3 \| \partial _3^{l - 1}{\nabla _h} \cdot {u_h} \|_{L^2}^{\frac{1}{2}} \| {\partial _3}\partial _3^{l - 1}{\nabla _h} \cdot {u_h}\|_{L^2}^{\frac{1}{2}}
 \| \partial _3^{3 - l} \partial _3\theta\|_{L^2}^{\frac{1}{2}} \| {\partial _1}\partial _3^{3- l}\partial _3 \theta \|_{L^2}^{\frac{1}{2}} \| \partial _3^3 \theta\|_{L^2}^{\frac{1}{2}} \| {\partial _2}\partial _3^3 \theta\|_{L^2}^{\frac{1}{2}} \\
\leqslant & C(\| u \|_{H^3} ^{\frac{1}{2}} \| {\nabla _h}u \|_{H^3} ^{\frac{1}{2}} \| {\nabla _h}\theta\|_{H^3} ^{\frac{3}{2}} \| \theta\|_{H^3} ^{\frac{1}{2}}
+\| {\nabla _h}u \|_{H^3}\| \theta\|_{H^3} \| {\nabla _h}\theta\|_{H^3} )\\
\leqslant & C(\| u \|_{H^3} + \| \theta  \|_{H^3})(\| {\nabla _h}u \|_{H^3}^2 + \| {\nabla _h}\theta  \|_{H^3}^2).
\end{aligned}
\end{equation*}
Collecting the above bound for $I_2$, we obtain
\begin{equation}\label{w-13}
K_{2}\leqslant C(\| u\|_{H^{3} }+\|\theta\|_{H^3})(\|\nabla _{h} u\|_{H^{3} }^{2}+\| {\nabla _h}\theta\|_{H^3}^2).
\end{equation}
Inserting (\ref{w-12}) and (\ref{w-13}) in (\ref{w-11}), integrating in time over $[0,t]$ and adding (\ref{w-5}), we deduce
\begin{equation*}
\begin{aligned}
E(t)&\leqslant E(0)+C\int_0^t \|u\|_{H^3}\|\nabla_hu\|_{H^3}^2 d\tau+C\int_0^t(\|u\|_{H^3}+\|\theta\|_{H^3})(\| {{\nabla _h}u} \|_{{H^3}}^2 + \|  {\nabla _h}\theta  \|_{{H^3}}^2) d\tau\\
&\leqslant E(0)+C_0E^{\frac{3}{2}}(t).\\
\end{aligned}
\end{equation*}
According to the bootstrapping argument, we assume that the initial data $\|(u_{0},\theta_{0})\|_{H^{3}}\leqslant\varepsilon:=\frac{1}{4C_{0}},$ namely
$$ E(0)\leqslant\frac{1}{16C_{0}^{2}}.$$
In fact, if we make the ansatz that,
$$ E(t)\leqslant\frac{1}{4C_{0}^{2}}.$$
Then (\ref{w-0}) implies
\begin{equation}
\begin{aligned}
E(t) & \leqslant {C_0}{E^{{\textstyle{1 \over 2}}}}(t) \cdot E(t) + E(0)\\
& \leqslant E(0) + \frac{1}{2}E(t).
\end{aligned}
\nonumber
\end{equation}
Consequently,
\begin{equation*}
\begin{aligned}
E(t) \leqslant\frac{1}{8C_{0}^{2}}.
\end{aligned}
\end{equation*}
The bootstrapping argument then implies the desire global bound on the solution $(u,\theta)$. As a consequence, we obtain the global existence of solutions. The uniqueness is obvious due to the high regularity of the solution. This completes the proof of Theorem \ref{T1.1}.

\section{Solution formula for the linearized problem}\label{S2}

In this section we present a solution formula for the linearized problem in terms of Fourier transform. By the change of the variables in the
system \eqref{E1.9}, linearized initial-boundary value problem around $(u^{(0)}, P^{(0)}, \Theta^{(0)})$ is written as
 \begin{eqnarray}
\begin{cases}
\p_t\bar{u}- \Delta_h \bar{u}-\bar{\theta} e_3+\nabla \varphi=0, &x\in\br^3_+,t>0,\\
\p_t\bar{\theta}- \Delta_h\bar{\theta}+\bar{u}_3=0,&x\in\br^3_+,t>0,\\
\nabla\cdot \bar{u}=0,&x\in\br^3_+,t>0,\\
(\bar{u}(x,0), \bar{\theta}(x,0))=(u_0(x),\theta_0(x))\equiv v_0, &x\in\br_+^3,\\
\frac{\partial \bar{u}_1}{\partial x_3}=\frac{\partial \bar{u}_2}{\partial x_3}=0, \bar{u}_3=0,\bar{\theta}=0,&~\mbox{on}~ x_3=0,t>0.
\end{cases}\label{E2.1}
\end{eqnarray}

Denote the solution $ \left(
                              \begin{array}{c}
                                \bar{u}(x,t) \\
                                \bar{\theta}(x,t)\\
                              \end{array}
                            \right)
$ of \eqref{E2.1}  by
$$V(t)v_0=\left(
         \begin{array}{c}
          U(t)u_0\\
          \Psi(t)\theta_0 \\
         \end{array}
       \right)=\left(
                 \begin{array}{c}
                   U_1(x,t) \\
                  U_2(x,t) \\
                   U_3(x,t) \\
                   \Psi(x,t) \\
                 \end{array}
               \right).
$$
The required estimates for $V(t)v_0$ will be obtained by its Fourier transform in $x_h=(x_1,x_2)$ and Fourier cosine and sine transforms in $x_3;$ to end it we introduce some notations. We  denote by $\hat{u}(\xi_h,x_3)(\xi_h=(\xi_1,\xi_2))$ the Fourier transform of $u(x_h,x_3)$ in $x_h=(x_1,x_2)\in\br^2$, and by $\mathcal{F}_c\hat{u}(\xi_h,\xi_3)(=\hat{u}_c(\xi_h,\xi_3))$ and $\mathcal{F}_s\hat{u}(\xi_h,\xi_3)(=\hat{u}_s(\xi_h,\xi_3))$ the Fourier cosine and sine transforms  in $x_3$, respectively:
\begin{eqnarray*}
&&\hat{u}(\xi_h,x_3)=\int_{\br^2}u(x_h,x_3)e^{-\ri \xi_h\cdot x_h}\dd x_h \\
&&\mathcal{F}_c\hat{u}(\xi_h,\xi_3)=\hat{u}_c(\xi_h,\xi_3)=\int_0^\infty\int_{\br^2}u(x_h,x_3)e^{-\ri \xi_h\cdot x_h}\cos\xi_3x_3\dd x_h\dd x_3 \\
&&\mathcal{F}_s\hat{u}(\xi_h,\xi_3)=\hat{u}_s(\xi_h,\xi_3)=\int_0^\infty\int_{\br^2}u(x_h,x_3)e^{-\ri \xi_h\cdot x_h}\sin\xi_3x_3\dd x_h\dd x_3.
\end{eqnarray*}

For the sake of clarity, we easily check the following properties of Fourier cosine and sine transforms (see \cite{Kob}).

\begin{lem}\label{lem1}
Let $\hat{u}_c,\hat{u}_s$ be the Fourier cosine and sine transforms of the functions $u$, then one has got\\

(1)  $\mathcal{F}_k(\widehat{\p_{x_h} u}(\xi_h,\xi_3))=\ri\xi_h \hat{u}_k, \mathcal{F}_k(\widehat{\Lambda^\lambda u}(\xi_h,\xi_3))=|\xi_h |^\lambda \hat{u}_k,(k=c,s,\lambda\in\br)$;\\

(2) $\mathcal{F}_c(\p_3 u(x_3))=\xi_3\hat{u}_s(\xi_3)-u(0), \mathcal{F}_s(\p_3 u(x_3))=-\xi_3\hat{u}_c(\xi_3)$;\\

(3) For any $u\in L^2(\br_+^3)$, $$\|u\|_{L^2}^2=(2\pi^3)^{-1}\|\hat{u}_c\|^2_{L^2}=(2\pi^3)^{-1}\|\hat{u}_s\|^2_{L^2};$$

(4) For any $u\in H^1(\br_+^3)$, $$\|\p_{x_h}u\|_{L^2}^2=(2\pi^3)^{-1}\|\widehat{\p_{x_h}u}_k\|^2_{L^2}=(2\pi^3)^{-1}\|\xi_h\hat{u}_k\|^2_{L^2},(k=c,s)$$ and
$$\|\p_{x_3}u\|_{L^2}^2=(2\pi^3)^{-1}\|\widehat{\p_{x_3}u}_s\|^2_{L^2}=(2\pi^3)^{-1}\|\xi_3\hat{u}_c\|^2_{L^2},$$
however, for any $u\in H_0^1(\br_+^3)$
$$\|\p_{x_3}u\|_{L^2}^2=(2\pi^3)^{-1}\|\widehat{\p_{x_3}u}_c\|^2_{L^2}=(2\pi^3)^{-1}\|\xi_3\hat{u}_s\|^2_{L^2}.$$
\end{lem}

We define $\mathcal{F}U$ for $U=U(x_h,x_3)=(u_1(x_h,x_3),u_2(x_h,x_3),u_3(x_h,x_3),\theta(x_h,x_3))$ by $$\mathcal{F}U(\xi_h,\xi_3)=\left(
                                       \begin{array}{c}
                                         \mathcal{F}_c \hat{u}_1(\xi_h,\xi_3) \\
                                         \mathcal{F}_c \hat{u}_2(\xi_h,\xi_3) \\
                                         \mathcal{F}_s \hat{u}_3(\xi_h,\xi_3) \\
                                         \mathcal{F}_s \hat{\theta}(\xi_h,\xi_3) \\
                                       \end{array}
                                     \right)
$$
Applying the transform $\mathcal{F}$ to the problem \eqref{E2.1} and Lemma \ref{lem1}, one obtain the following system
\begin{eqnarray}
&\p_t\mathcal{F}_c\hat{u}_h+|\xi_h|^2 \mathcal{F}_c\hat{u}_h+ \widehat{\nabla_h \varphi}_c=0, &\xi\in\br^3_+,t>0,\label{E2.2}\\
&\p_t\mathcal{F}_s\hat{u}_3+|\xi_h|^2 \mathcal{F}_s\hat{u}_3-\hat{\theta}_s+ \widehat{\p_3 \varphi}_s=0, &\xi\in\br^3_+,t>0,\label{E2.3}\\
&\p_t\hat{\theta}_s+|\xi_h|^2\hat{\theta}_s+\mathcal{F}_s\hat{u}_3=0, &\xi\in\br^3_+,t>0. \label{E2.4}\\
&\ri\xi_h\cdot \mathcal{F}_c\hat{u}_h  +\xi_3\mathcal{F}_s\hat{u}_3=0,&\xi\in\br^3_+,t>0,\label{E2.5}
\end{eqnarray}
To  present the formula of the pressure in above system, we observe that $\varphi$ satisfies the following problem
\begin{eqnarray}
\begin{cases}
   \Delta \varphi =\p_{x_3}{\theta},~ &\mbox{in}~\br_+^3 \\
   \p_3 \varphi=0,~&\mbox{on}~x_3=0.
\end{cases}\label{E2.6}
\end{eqnarray}
 Used the fourier cosine transform to both sides of the equation of \eqref{E2.6}, by the Lemma \ref{lem1} and the boundary conditions of \eqref{E2.6},  one refers
  \begin{equation*}
  -|\xi|^2\hat{\varphi}_c(\xi_h,\xi_3)=\xi_3\hat{\theta}_s,
  \end{equation*}
  then we can solve above equation as
  \begin{equation}\label{E2.7}
  \hat{\varphi}_c=-\frac{\xi_3}{|\xi|^2}\hat{\theta}_s.
  \end{equation}

  We substitute \eqref{E2.7} into \eqref{E2.2}-\eqref{E2.4} and obtain
  \begin{equation}\label{E2.8}
  \p_t\mathcal{F} U(t)=A\mathcal{F} U(t),
  \end{equation}
  where
  \begin{equation}\label{A2.8}
  A=\left(
              \begin{array}{cccc}
                -|\xi_h|^2 & 0 & 0 & \frac{\rm i\xi_1\xi_3}{|\xi|^2} \\
                0 &-|\xi_h|^2 & 0 & \frac{\rm i\xi_2\xi_3}{|\xi|^2} \\
                0 & 0 &-|\xi_h|^2 &  \frac{|\xi_h|^2}{|\xi|^2}\\
                0 & 0 & -1 & -|\xi_h|^2 \\
              \end{array}
            \right)
  \end{equation}
  By solving the ODEs \eqref{E2.8} with initial data $\mathcal{F}U_0$ as follows:
  \begin{eqnarray}
&&\mathcal{F}_c\hat{u}_h=e^{\lambda_1t}\mathcal{F}_c\hat{u}_{0h}-\frac{\ri \xi_h\xi_3}{|\xi_h|^2}\left(1-\cos\frac{|\xi_h|}{|\xi|}t\right)e^{\lambda_1t}\mathcal{F}_s\hat{u}_{03}+\frac{\ri \xi_h\xi_3}{|\xi||\xi_h|}\left(\sin\frac{|\xi_h|}{|\xi|}t\right) e^{\lambda_1t} \mathcal{F}_s\hat{\theta}_{0},\label{E2.9}\\
&&\mathcal{F}_s\hat{u}_3=\left(\cos\frac{|\xi_h|}{|\xi|}t\right) e^{\lambda_1t}\mathcal{F}_s\hat{u}_{03}
+\left(\frac{|\xi_h|}{|\xi|}\sin\frac{|\xi_h|}{|\xi|}t\right)e^{\lambda_1t}\mathcal{F}_s\hat{\theta}_{0},\label{E2.10}\\
&&\mathcal{F}_s \hat{\theta}=-\left(\frac{|\xi|}{|\xi_h|}\sin\frac{|\xi_h|}{|\xi|}t\right)e^{\lambda_1t}\mathcal{F}_s\hat{u}_{03}
+\left(\cos\frac{|\xi_h|}{|\xi|}t\right) e^{\lambda_1t}\mathcal{F}_s\hat{\theta}_{0}. \label{E2.11}
\end{eqnarray}
where $\lambda_1=-|\xi_h|^2.$

By the formula \eqref{E2.9}-\eqref{E2.11}, we shall verify the stability and optimal decay rates stated in the following proposition (see \cite{Kob}) for the linearized equations of \eqref{E2.1} . The results in this proposition and their proofs are part of our program for optimal decay rates on the nonlinear system (1.6), and will be used in the proof of our main result, Theorem \ref{T1.2} below.

\begin{prop}\label{prop2.1}
Let $s$ be non-negative and $\sigma>0.$ Assume that the initial data $v_0=(u_0,\theta_0)$ with the boundary conditions \eqref{N1.5} and \eqref{N1.6}  satisfies $\nabla\cdot u_0=0.$

\vspace{0.2cm}
(1) If $v_0\in \dot{H}^s,\Lambda^{-\sigma}v_0=(\Lambda^{-\sigma}u_0,\Lambda^{-\sigma}\theta_0)\in \dot{H}^s,$  then the solution $V(t)v_0$ given by \eqref{E2.9}-\eqref{E2.11} satisfies
\begin{equation}\label{R2.11}
\|V(t)v_0\|_{\dot{H}^s}\leqslant C(\|v_0\|_{\dot{H}^s}+\|\Lambda^{-\sigma}v_0\|_{\dot{H}^s})(1+t)^{-\frac{\sigma}{2}}.
\end{equation}

\vspace{0.2cm}
(2)  If $\p_3v_0\in \dot{H}^s$ and $\Lambda^{-\sigma}\p_3v_0 \in \dot{H}^s,$  then
\begin{equation}\label{R2.12}
\|\p_3V(t)v_0\|_{\dot{H}^s}\leqslant C(\|\p_3v_0\|_{\dot{H}^s}+\|\Lambda^{-\sigma}\p_3v_0\|_{\dot{H}^s})(1+t)^{-\frac{\sigma}{2}}.
\end{equation}

\vspace{0.2cm}
(3)  If   $\Lambda^{-\sigma}v_0 \in \dot{H}^s,$  then
\begin{equation}\label{R2.13}
\|\nabla_hV(t)v_0\|_{\dot{H}^s}\leqslant C\|\Lambda^{-\sigma}v_0\|_{\dot{H}^s}t^{-\frac{\sigma+1}{2}}.
\end{equation}
In addition, assume that $\Lambda^{-\sigma}\nabla_hv_0 \in \dot{H}^s,$
\begin{equation}\label{R2.14}
\|\nabla_hV(t)v_0\|_{\dot{H}^s}\leqslant C(\|\Lambda^{-\sigma}\nabla_hv_0\|_{\dot{H}^s}+\|\Lambda^{-\sigma}v_0\|_{\dot{H}^s})(1+t)^{-\frac{\sigma+1}{2}}.
\end{equation}
\end{prop}

\begin{proof} Due to the frequency decoupling in the solution $V(t)v_0$ given by \eqref{E2.9}-\eqref{E2.11}, it suffices to set $s=0$ and consider the $L^2-$norm. Start with the estimate of $U_h(x,t)$, the first term in \eqref{E2.9} is easily bounded. In fact, for any $0\leqslant t<1$,
\begin{equation}\label{E2.15}
\|e^{\lambda_1t}\mathcal{F}_c\hat{u}_{0h}\|_{L^2}\leqslant\|\mathcal{F}_c\hat{u}_{0h}\|_{L^2}=2\pi^3\|u_{0h}\|_{L^2}.
\end{equation}
For $t\ge 1,$
\begin{equation}\label{E2.16}
\|e^{\lambda_1t}\mathcal{F}_c\hat{u}_{0h}\|_{L^2}=\||\xi_h|^\sigma e^{-|\xi_h|^2t}|\xi_h|^{-\sigma}\mathcal{F}_c\hat{u}_{0h}\|_{L^2}\leqslant Ct^{-\frac{\sigma}{2}}\||\xi_h|^{-\sigma}\mathcal{F}_c\hat{u}_{0h}\|_{L^2}\leqslant Ct^{-\frac{\sigma}{2}}\|\Lambda_h^{-\sigma}u_{0h}\|_{L^2}.
\end{equation}
Combining \eqref{E2.15} and \eqref{E2.16} yields
\begin{equation*}
\|e^{\lambda_1t}\mathcal{F}_c\hat{u}_{0h}\|_{L^2} \leqslant C(\|u_{0h}\|_{L^2}+\|\Lambda_h^{-\sigma}u_{0h}\|_{L^2})(1+t)^{-\frac{\sigma}{2}}.
\end{equation*}
From above procedure, we find it is sufficient to consider the estimates for $t\ge 1.$  To dealt with the second term in \eqref{E2.9}, we can find that
\begin{equation*}
\begin{aligned}
\left\|\frac{\ri \xi_h\xi_3}{|\xi_h|^2}\left(1-\cos\frac{|\xi_h|}{|\xi|}t\right)e^{\lambda_1t}\mathcal{F}_s\hat{u}_{03}\right\|_{L^2}&\leqslant2\left\|\frac{\ri \xi_h}{|\xi_h|^2}e^{\lambda_1t}\xi_3\mathcal{F}_s\hat{u}_{03}\right\|_{L^2}\leqslant2\left\|\frac{1}{|\xi_h|}e^{\lambda_1t}\xi_h\cdot\mathcal{F}_c\hat{u}_{0h}\right\|_{L^2}\\
&\leqslant 2\left\|e^{\lambda_1t}\mathcal{F}_c\hat{u}_{0h}\right\|_{L^2}\leqslant C(\|u_{0h}\|_{L^2}+\|\Lambda_h^{-\sigma}u_{0h}\|_{L^2})(1+t)^{-\frac{\sigma}{2}}.
\end{aligned}
\end{equation*}

 It is not difficult to see that
$$\left\|\frac{\ri \xi_h\xi_3}{|\xi||\xi_h|}\left(\sin\frac{|\xi_h|}{|\xi|}t\right) e^{\lambda_1t} \mathcal{F}_s\hat{\theta}_{0}\right\|_{L^2}\leqslant C\left(\|e^{\lambda_1t} \mathcal{F}_s\hat{\theta}_{0})\|_{L^2}\right). $$
By the argument from the \eqref{E2.15} and \eqref{E2.16}, one obtains that
 \begin{equation*}
 \left\|\frac{\ri \xi_h\xi_3}{|\xi_h||\xi|}\left(\sin\frac{|\xi_h|}{|\xi|}t\right) e^{\lambda_1t} \mathcal{F}_s\hat{\theta}_{0}\right\|_{L^2}\leqslant C\left(\|\theta_0\|_{L^2}+\|\Lambda^{-\sigma}\theta_0\|_{L^2}\right)(1+t)^{-\frac{\sigma}{2}}.
 \end{equation*}
 Therefore, we obtain the following estimate
$$\|U_h(x,t)\|_{L^2}\leqslant C(\|v_0\|_{L^2}+\|\Lambda^{-\sigma}v_0\|_{L^2})(1+t)^{-\frac{\sigma}{2}}.$$
Applying the formula \eqref{E2.10} and the argument from the \eqref{E2.15} and \eqref{E2.16} again, we know that
\begin{equation*}
\begin{aligned}
\|\mathcal{F}_s\hat{u}_3\|_{L^2}&=\left\|\left(\cos\frac{|\xi_h|}{|\xi|}t\right) e^{\lambda_1t}\mathcal{F}_s\hat{u}_{03}
+\left(\frac{|\xi_h|}{|\xi|}\sin\frac{|\xi_h|}{|\xi|}t\right)e^{\lambda_1t}\mathcal{F}_s\hat{\theta}_{0}\right\|_{L^2}\\
&\leqslant \left(\|e^{\lambda_1t}\mathcal{F}_s\hat{u}_{03}\|_{L^2}+\|e^{\lambda_1t} \mathcal{F}_s\hat{\theta}_{0})\|_{L^2}\right)\\
&\leqslant C\left(\|(u_{03},\theta_0)\|_{L^2}+\|\Lambda^{-\sigma}(u_{03},\theta_0)\|_{L^2}\right)(1+t)^{-\frac{\sigma}{2}}
\end{aligned}
\end{equation*}
Since the factor $\frac{|\xi|}{|\xi_h|}\geqslant 1,$  there is a subtle difference   for the decay of $\hat{\theta}_s$. By the incompressibility \eqref{E2.6}, we observe that
\begin{equation*}
\begin{aligned}
\|I_1\|_{L^2}&=\left\|\frac{|\xi|}{|\xi_h|}\sin\frac{|\xi_h|}{|\xi|}t e^{\lambda_1t}\mathcal{F}_s\hat{u}_{03}\right\|_{L^2}\\
&\leqslant\left\|\frac{|\xi_h|+|\xi_3|}{|\xi_h|}\sin\frac{|\xi_h|}{|\xi|}t e^{\lambda_1t}\mathcal{F}_s\hat{u}_{03}\right\|_{L^2}\\
&\leqslant\|e^{\lambda_1t}\mathcal{F}_s\hat{u}_{03}\|_{L^2}+\left\|e^{\lambda_1t}\frac{1}{|\xi_h|}|\xi_h\cdot\mathcal{F}_c\hat{u}_{0h}| \right\|_{L^2}\\
&\leqslant C(\|u_0\|_{L^2}+\|\Lambda_h^{-\sigma}u_0\|_{L^2})(1+t)^{-\frac{\sigma}{2}}.
\end{aligned}
\end{equation*}
It is easy to see that
\begin{equation*}
\begin{aligned}
\|\hat{\theta}_s\|&\leqslant \|I_1\|_{L^2}+\| \mathcal{F}_s\hat{\theta}_{0}\|_{L^2}\\
&\leqslant C(\|v_0\|_{L^2}+\|\Lambda_h^{-\sigma}v_0\|_{L^2})(1+t)^{-\frac{\sigma}{2}}.
\end{aligned}
\end{equation*}
 By the Lemma \ref{lem1} and these estimates, one has finished the proof of \eqref{R2.11}.  In order to verify the \eqref{R2.12}, from Lemma\ref{lem1}, it suffices to check the estimate for $\|\xi_3\mathcal{F}U\|_{L^2}$. Actually, these estimates  are quite similar to those for \eqref{R2.11} except for two terms, i.e. $I_2=\left\|\frac{\ri \xi_h\xi^2_3}{|\xi_h|^2}\left(1-\cos\frac{|\xi_h|}{|\xi|}t\right)e^{\lambda_1t}\mathcal{F}_s\hat{u}_{03}\right\|_{L^2}$ and $I_3=\left\|\frac{|\xi|\xi_3}{|\xi_h|}\sin\frac{|\xi_h|}{|\xi|}t e^{\lambda_1t}\mathcal{F}_s\hat{u}_{03}\right\|_{L^2}$.  It is similar between the estimate of $I_2$ and the one of $I_3$. From the incompressibility and Lemma \ref{lem1}, one infers
 \begin{equation*}
 \begin{aligned}
 I_2=\left\|\frac{\ri \xi_h\xi^2_3}{|\xi_h|^2}\left(1-\cos\frac{|\xi_h|}{|\xi|}t\right)e^{\lambda_1t}\mathcal{F}_s\hat{u}_{03}\right\|_{L^2}\leqslant 2\left\|\frac{ \xi_h\xi_3}{|\xi_h|^2}e^{\lambda_1t}\xi_h\cdot\mathcal{F}_c\hat{u}_{0h}\right\|_{L^2}\\
 \leqslant 2\left\| \xi_3 e^{\lambda_1t} \mathcal{F}_c\hat{u}_{0h}\right\|_{L^2}\leqslant C \left\|e^{\lambda_1t} \mathcal{F}_s\widehat{\p_3u}_{0h}\right\|_{L^2}\\
 \leqslant C(\|\p_3u_{0h}\|_{L^2}+\|\Lambda^{-\sigma}\p_3u_{0h}\|)(1+t)^{-\frac{\sigma}{2}}.
 \end{aligned}
 \end{equation*}
 Therefore, we have checked  \eqref{R2.12}.

  Noticing that it is easy to obtain the extra decay factor by the following inequality, for any $t>0,$
 $$\| |\xi_h|e^{-|\xi_h|^2t}\mathcal{F}f\|_{L^2}\leqslant Ct^{-\frac{\sigma+1}{2}}\|\Lambda^{-\sigma}_hf\|_{L^2}.$$
 This obtains \eqref{R2.13}. Combining \eqref{R2.13} for $t\geqslant 1$ and the basic inequality with $0\leqslant t<1,$
 $$\| |\xi_h|e^{-|\xi_h|^2t}f\|\leqslant C\| |\xi_h|f \|_{L^2}$$
 leads to \eqref{R2.14}. This completes the proof.
\end{proof}

\section{Decays of solutions to the nonlinear system}\label{S4}

In this section, we obtain the decays of solutions to nonlinear system. To reach it, we need some technique lemmas.
\begin{lem}\label{lem2}
Let $2\leqslant p<\infty$ and $s>\frac{1}{2}-\frac{1}{p}$. Then, there exists a constant $C=C(p,s)$ such that, for any $f\in H^s(\br)$,
\begin{equation}\label{I4.1}
\|f\|_{L^p}\leqslant C\|f\|^{1-\frac{1}{s}\left(\frac{1}{2}-\frac{1}{p}\right)}_{L^2}\|\Lambda^s f\|^{\frac{1}{s}\left(\frac{1}{2}-\frac{1}{p}\right)}_{L^2}.
\end{equation}
\end{lem}
Next lemma provides an exact $L^p-L^q$ decay estimate for the generalized heat operator associated with a fractional Laplacian (see\cite{JYWu,Wu}).
\begin{lem}\label{lem3}
Let $\beta\geqslant 0, \alpha>0, 1\leqslant p\leqslant q\leqslant\infty.$ Then
\begin{equation}\label{I4.3}
\|\Lambda^\beta e^{-(-\Delta)^\alpha t}f\|_{L^q(\br^d)}\leqslant C t^{-\frac{\beta}{2\alpha}-\frac{d}{2\alpha}\left(\frac{1}{p}-\frac{1}{q}\right)}\|f\|_{L^p(\br^d)}.
\end{equation}
\end{lem}
\textbf{Proof of Theorem 1.2} The bootstrapping argument is suitable for our purpose. We assume the initial datum $v_0=(u_0,\theta_0)$ satisfies  \eqref{I1.6},\eqref{I1.7} and \eqref{I1.8} for $\varepsilon>0$ small enough. The bootstrapping argument start with the ansatz that, for a suitably chosen $C_0>0,$
\begin{eqnarray}
&&\|u(t)\|_{H^3 }+\|\theta(t)\|_{H^3 }\leqslant C_0\varepsilon, \label{I4.4}\\
&&\|\Lambda_h^{-\sigma} u(t)\|_{L^2}+\|\Lambda_h^{-\sigma}\theta(t)\|_{L^2}\leqslant C_0\varepsilon, \label{I4.5}\\
&&\|u(t)\|_{L^2},\|\theta(t)\|_{L^2}\leqslant C_0\varepsilon (1+t)^{-\frac{\sigma}{2}+\delta},\label{I4.6}\\
&&\|\partial_3u_h(t)\|_{L^2},\|\partial_3\theta(t)\|_{L^2}\leqslant C_0\varepsilon (1+t)^{-\frac{\sigma}{2}+3\delta},\label{I4.7}\\
&&\|\nabla_h u(t)\|_{L^2},\|\nabla_h\theta(t)\|_{L^2}\leqslant C_0\varepsilon (1+t)^{-\frac{\sigma+1}{2}+\delta},\label{I4.8}
\end{eqnarray}
for $t\in [0,T]$ with some $T>0.$ The inequalities from \eqref{I4.4} to \eqref{I4.8} hold on the initial time interval $[0,T]$ by local existence. We shall show that \eqref{I4.4}, \eqref{I4.5},\eqref{I4.6}, \eqref{I4.7} and \eqref{I4.8} remain true with $C_0$ replaced by ${C_0}/{2},$ namely
\begin{eqnarray}
&&\|u(t)\|_{H^3 }+\|\theta(t)\|_{H^3 }\leqslant \frac{C_0}{2}\varepsilon, \label{I4.9}\\
&&\|\Lambda_h^{-\sigma} u(t)\|_{L^2}+\|\Lambda_h^{-\sigma}\theta(t)\|_{L^2}\leqslant \frac{C_0}{2}\varepsilon, \label{I4.10}\\
&&\|u(t)\|_{L^2},\|\theta(t)\|_{L^2}\leqslant \frac{C_0}{2}\varepsilon (1+t)^{-\frac{\sigma}{2}+\delta},\label{I4.11}\\
&&\|\partial_3u(t)\|_{L^2},\|\partial_3\theta(t)\|_{L^2}\leqslant \frac{C_0}{2}\varepsilon (1+t)^{-\frac{\sigma}{2}+3\delta},\label{I4.12}\\
&&\|\nabla_h u(t)\|_{L^2},\|\nabla_h\theta(t)\|_{L^2}\leqslant \frac{C_0}{2}\varepsilon (1+t)^{-\frac{\sigma+1}{2}+\delta}.\label{I4.13}
\end{eqnarray}
The bootstrapping argument then assert that inequalities from \eqref{I4.9} to \eqref{I4.13} hold for all $t>0.$ In the proof of Theorem \ref{T1.1}, we check that
$$\|u(t)\|_{H^2 }+\|\theta(t)\|_{H^2 }\leqslant C_1\varepsilon.$$
Then the inequality \eqref{I4.9} holds as one takes $C_0\geqslant 2C_1.$

 Applying the same argument of subsection 5.1 in \cite{JYWu}, it is easy to obtain the following bound
 \begin{equation*}
 \begin{aligned}
 &\|\Lambda^{-\sigma}_h u\|_{L^2}^2+\|\Lambda^{-\sigma}_h\theta\|_{L^2}^2+2\int_0^t\left(\|\Lambda^{1-\sigma}_h u\|_{L^2}^2+\|\Lambda^{1-\sigma}_h\theta\|_{L^2}^2\right)\dd\tau \\
 &\leqslant C\int_0^t\|u_3\|_{L^2}^{\sigma-\frac{1}{2}}\|\nabla_h u\|_{L^2}^{\frac{3}{2}-\sigma}\|\p_3 u_h\|_{L^2}\|\Lambda^{-\sigma}u\|_{L^2}\dd\tau\\
 &+ C\int_0^t\|u_h\|_{L^2}^{\sigma-\frac{1}{2}}\|\nabla_h u\|_{L^2}^{2-\sigma}\|\p_3 u_h\|_{L^2}^\frac{1}{2}\|\Lambda^{-\sigma}u\|_{L^2}\dd\tau\\
 &+C\int_0^t\|u_3\|_{L^2}^{\sigma-\frac{1}{2}}\|\nabla_h u\|_{L^2}^{\frac{5}{2}-\sigma}\|\Lambda^{-\sigma}u\|_{L^2}\dd\tau\\
 &+C\int_0^t\|u_h\|_{L^2}^{\sigma-\frac{1}{2}}\|\nabla_h u\|_{L^2}^{1-\sigma}\|\p_3 u_h\|_{L^2}^\frac{1}{2}\|\nabla_h \theta\|_{L^2}\|\Lambda^{-\sigma}\theta\|_{L^2}\dd\tau\\
 &+C\int_0^t\|u_3\|_{L^2}^{\sigma-\frac{1}{2}}\|\nabla_h u\|_{L^2}^{\frac{3}{2}-\sigma}\|\p_3 \theta\|_{L^2}\|\Lambda^{-\sigma}\theta\|_{L^2}\dd\tau\\
 & +\|\Lambda^{-\sigma}_h u_0\|_{L^2}^2+\|\Lambda^{-\sigma}_h\theta_0\|_{L^2}^2.
 \end{aligned}
 \end{equation*}
 By the bounds in \eqref{I4.5}-\eqref{I4.8}, It infers that
 \begin{equation*}
 \begin{aligned}
 &\|\Lambda^{-\sigma}_h u\|_{L^2}^2+\|\Lambda^{-\sigma}_h\theta\|_{L^2}^2+2\int_0^t\left(\|\Lambda^{1-\sigma}_h u\|_{L^2}^2+\|\Lambda^{1-\sigma}_h\theta\|_{L^2}^2\right)\dd\tau \\
 &\leqslant CC_0^3\varepsilon^3\int_0^t(1+\tau)^{-\frac{\sigma}{2}-\frac{3}{4}+4\delta}\dd\tau\\
 &+ CC_0^3\varepsilon^3\left(\int_0^t(1+\tau)^{-\frac{\sigma}{2}-1+\delta}\dd\tau+\int_0^t(1+\tau)^{-\frac{\sigma}{2}-\frac{5}{4}+2\delta}\dd\tau\right)\\
& +\|\Lambda^{-\sigma}_h u_0\|_{L^2}^2+\|\Lambda^{-\sigma}_h\theta_0\|_{L^2}^2.
 \end{aligned}
 \end{equation*}

 As $\frac{3}{4}\leqslant\sigma<1$ and $$0\leqslant\delta<\frac{\sigma}{8}-\frac{1}{16},$$ we easily claim the inequality \eqref{I4.10}holds  while taking $\varepsilon>0$ small enough. The rest of the section is divided into three subsections,  we   verify the inequality \eqref{I4.11}, \eqref{I4.12} and \eqref{I4.13} in the first subsection, the second subsection, the third subsection respectively.

\subsection{Estimates of $\|u\|_{L^2},\|\theta\|_{L^2}$ and verification of \eqref{I4.11} }

To verify the inequality \eqref{I4.11}, we take advantage of Duhamel's principle to obtain the formula of the solution to the nonlinear system.
\begin{prop}\label{prop4.1}
The system in \eqref{E1.9} can be converted into the following form
\begin{eqnarray}
&\mathcal{F}_c\widehat{u}_h&=e^{\lambda_1t}\mathcal{F}_c\widehat{u}_{0h}-\frac{\ri \xi_h\xi_3}{|\xi_h|^2}\left(1-\cos\frac{|\xi_h|}{|\xi|}t\right)e^{\lambda_1t}\mathcal{F}_s\widehat{u}_{03}+\frac{\ri \xi_h\xi_3}{|\xi||\xi_h|}\left(\sin\frac{|\xi_h|}{|\xi|}t\right) e^{\lambda_1t} \mathcal{F}_s\widehat{\theta}_{0}\nonumber\\
&&-\int_0^te^{\lambda_1(t-\tau)}\mathcal{F}_c(\widehat{\mathbb{P}u\cdot\nabla u})_h(\tau)\dd\tau \nonumber\\
&&+\int_0^t\frac{\ri \xi_h\xi_3}{|\xi_h|^2}\left(1-\cos\frac{|\xi_h|}{|\xi|}(t-\tau)\right)e^{\lambda_1(t-\tau)}\mathcal{F}_s(\widehat{\mathbb{P}u\cdot\nabla u})_3(\tau)\dd \tau \nonumber\\
&&-\int_0^t\frac{\ri \xi_h\xi_3}{|\xi||\xi_h|}\left(\sin\frac{|\xi_h|}{|\xi|}(t-\tau)\right) e^{\lambda_1(t-\tau)} \mathcal{F}_s\widehat{u\cdot\nabla\theta}(\tau)\dd \tau \label{E4.14}\\
&\mathcal{F}_s\widehat{u}_3&=\left(\cos\frac{|\xi_h|}{|\xi|}t\right) e^{\lambda_1t}\mathcal{F}_s\hat{u}_{03}
+\left(\frac{|\xi_h|}{|\xi|}\sin\frac{|\xi_h|}{|\xi|}t\right)e^{\lambda_1t}\mathcal{F}_s\hat{\theta}_{0},\nonumber\\
&&-\int_0^t\left(\cos\frac{|\xi_h|}{|\xi|}(t-\tau)\right) e^{\lambda_1(t-\tau)}\mathcal{F}_s(\widehat{\mathbb{P}u\cdot\nabla u})_3(\tau)\dd \tau \nonumber\\
&&-\int_0^t\left(\frac{|\xi_h|}{|\xi|}\sin\frac{|\xi_h|}{|\xi|}(t-\tau)\right)e^{\lambda_1(t-\tau)}\mathcal{F}_s\widehat{u\cdot\nabla\theta}(\tau)\dd \tau \label{E4.15}\\
&\mathcal{F}_s \hat{\theta}&=-\left(\frac{|\xi|}{|\xi_h|}\sin\frac{|\xi_h|}{|\xi|}t\right)e^{\lambda_1t}\mathcal{F}_s\hat{u}_{03}
+\left(\cos\frac{|\xi_h|}{|\xi|}t\right) e^{\lambda_1t}\mathcal{F}_s\hat{\theta}_{0} \nonumber\\
&&+\int_0^t\left(\frac{|\xi|}{|\xi_h|}\sin\frac{|\xi_h|}{|\xi|}(t-\tau)\right)e^{\lambda_1(t-\tau)}\mathcal{F}_s(\widehat{\mathbb{P}u\cdot\nabla u})_3(\tau)\dd \tau \nonumber\\
&&- \int_0^t\left(\cos\frac{|\xi_h|}{|\xi|}(t-\tau)\right) e^{\lambda_1(t-\tau)}\mathcal{F}_s\widehat{u\cdot\nabla\theta}(\tau)\dd \tau \label{E4.16}
\end{eqnarray}
where $\lambda_1=-|\xi_h|^2.$
\end{prop}
\begin{proof}
Set $$\hat{F}(\tau)=\left(
                \begin{array}{c}
                  \mathcal{F}_c(\widehat{\mathbb{P}u\cdot\nabla u})_h(\tau) \\
                  \mathcal{F}_s(\widehat{\mathbb{P}u\cdot\nabla u})_3(\tau) \\
                  \mathcal{F}_s\widehat{u\cdot\nabla\theta}(\tau)\\
                \end{array}
              \right)
$$
by the Duhamel's principle,
$$\mathcal{F}U(\xi_h,\xi_3,t)=\left(
                                       \begin{array}{c}
                                         \mathcal{F}_c \hat{u}_1(\xi_h,\xi_3,t) \\
                                         \mathcal{F}_c \hat{u}_2(\xi_h,\xi_3,t) \\
                                         \mathcal{F}_s \hat{u}_3(\xi_h,\xi_3,t) \\
                                         \mathcal{F}_s \hat{\theta}(\xi_h,\xi_3,t) \\
                                       \end{array}
                                     \right)=e^{At}\mathcal{F}v_0+\int_0^te^{A(t-\tau)}\hat{F}(\tau)\dd\tau.$$
Substituting the matrix $A$ in \eqref{A2.8} into the formula above, we easily obtain the solutions represented by \eqref{E4.14},\eqref{E4.15} and \eqref{E4.16}.
\end{proof}
In the representation of solutions, we must dealt with the terms $\mathbb{P}(u\cdot\nabla u)$. To do it, by the Helmholtz decomposition Theorem, we know that $\mathbb{P}(u\cdot\nabla u)=u\cdot\nabla u-\nabla\psi.$ Since $u\cdot\nabla u_3|_{x_3=0}=0,$  it implies that  $\psi$ satisfies the following problem
\begin{eqnarray}
\begin{cases}
\Delta\psi=\di(u\cdot\nabla u), &x_3>0,\\
\frac{\p\psi}{\p x_3}=0, &\text{on}~ x_3=0.
\end{cases}\label{E4.17}
\end{eqnarray}
Apply Fourier cosine transform to \eqref{E4.17}, we can find that
\begin{equation}\label{E4.18}
\widehat{\psi}_c=-\frac{\ri \xi_h}{|\xi|^2}\cdot\mathcal{F}_c(\widehat{u\cdot\nabla u_h})-\frac{\xi_3}{|\xi|^2}\mathcal{F}_s(\widehat{u\cdot\nabla u_3})
\end{equation}
From \eqref{E4.18}, we   calculate that
\begin{eqnarray}
&\mathcal{F}_c(\widehat{\mathbb{P}u\cdot\nabla u})_h&=\mathcal{F}_c(\widehat{u\cdot\nabla u_h})-\sum\limits_{k=1}^2\frac{ \xi_h\xi_k}{|\xi|^2}\mathcal{F}_c(\widehat{u\cdot\nabla u_k})+\frac{\ri\xi_h\xi_3}{|\xi|^2}\mathcal{F}_s(\widehat{u\cdot\nabla u_3}),\label{E4.19}\\
&\mathcal{F}_s(\widehat{\mathbb{P}u\cdot\nabla u})_3&=\frac{|\xi_h|^2}{|\xi|^2}\mathcal{F}_s(\widehat{u\cdot\nabla u_3})-\frac{ \ri \xi_3\xi_h}{|\xi|^2}\cdot\mathcal{F}_c(\widehat{u\cdot\nabla u_h})\label{E4.20}
\end{eqnarray}
Now we check the estimates of $\|u\|_{L^2}$ and $\|\theta\|_{L^2}$. In Proposition \ref{prop4.1},  observed that the linear parts in the formula of solutions have been estimated in Proposition \ref{prop2.1}, it is sufficient to check the integral parts. For brevity, we denote these parts in \eqref{E4.14}-\eqref{E4.16} as
\begin{eqnarray*}
&J_1&=\int_0^te^{\lambda_1(t-\tau)}\mathcal{F}_c(\widehat{\mathbb{P}u\cdot\nabla u})_h(\tau)\dd\tau,\\
&J_2&=\int_0^t\frac{\ri \xi_h\xi_3}{|\xi_h|^2}\left(1-\cos\frac{|\xi_h|}{|\xi|}(t-\tau)\right)e^{\lambda_1(t-\tau)}\mathcal{F}_s(\widehat{\mathbb{P}u\cdot\nabla u})_3(\tau)\dd \tau,\\
&J_3&=\int_0^t\frac{\ri \xi_h\xi_3}{|\xi||\xi_h|}\left(\sin\frac{|\xi_h|}{|\xi|}(t-\tau)\right) e^{\lambda_1(t-\tau)} \mathcal{F}_s\widehat{u\cdot\nabla\theta}(\tau)\dd \tau,\\
&J_4&=\int_0^t\left(\cos\frac{|\xi_h|}{|\xi|}(t-\tau)\right) e^{\lambda_1(t-\tau)}\mathcal{F}_s(\widehat{\mathbb{P}u\cdot\nabla u})_3(\tau)\dd \tau, \\
&J_5&=\int_0^t\left(\frac{|\xi_h|}{|\xi|}\sin\frac{|\xi_h|}{|\xi|}(t-\tau)\right)e^{\lambda_1(t-\tau)}\mathcal{F}_s\widehat{u\cdot\nabla\theta}(\tau)\dd \tau,\\
&J_6&=\int_0^t\left(\frac{|\xi|}{|\xi_h|}\sin\frac{|\xi_h|}{|\xi|}(t-\tau)\right)e^{\lambda_1(t-\tau)}\mathcal{F}_s(\widehat{\mathbb{P}u\cdot\nabla u})_3(\tau)\dd \tau,\\
&J_7&=\int_0^t\left(\cos\frac{|\xi_h|}{|\xi|}(t-\tau)\right) e^{\lambda_1(t-\tau)}\mathcal{F}_s\widehat{u\cdot\nabla\theta}(\tau)\dd \tau.
\end{eqnarray*}
 Noticing that the structure of the formula above, we can classify four kinds, namely, one kind is $J_1$, the second kind is $J_2,J_6$  together, the third kind is collected by $J_3,J_5,J_7$ and the last kind is $J_4.$ Denote $U(t)v_0$ by the solution of the system \eqref{E1.9} with boundary conditions \eqref{N1.5} and \eqref{N1.6} , by the Lemma \ref{lem1},
  $$\|U(t)v_0\|_{L^2}\leqslant C\|\mathcal{F}(U(t)v_0)\|_{L^2}\leqslant \|e^{At}\mathcal{F}(v_0)\|_{L^2}+\sum\limits_{k=1}^7\|J_k\|_{L^2}.$$
From the \eqref{E4.19} and Minkowski's inequality, we have got
\begin{equation*}
\begin{aligned}
\|J_1\|_{L^2}&\leqslant \int_0^t\|e^{\lambda_1 (t-\tau)}\mathcal{F}_c\widehat{\mathbb{P}(u\cdot\nabla u)_h}\|_{L^2}\dd\tau\\
&\leqslant 2\int_0^t\|e^{\lambda_1 (t-\tau)}\mathcal{F}_c\widehat{(u\cdot\nabla u_h)}\|_{L^2}\dd\tau+\int_0^t \|e^{\lambda_1 (t-\tau)}\mathcal{F}_s\widehat{(u\cdot\nabla u_3)}\|_{L^2}\dd\tau\\
&:=2J_{11}+J_{12}.
\end{aligned}
\end{equation*}
Since $\lambda_1=-|\xi_h|,$ then
\begin{equation*}
\begin{aligned}
 J_{11} &= \int_0^t\|e^{\lambda_1 (t-\tau)}\mathcal{F}_c\widehat{(u\cdot\nabla u_h)}\|_{L^2}\dd\tau\\
 &=\int_0^t\|e^{\Delta_h(t-\tau)}u_h\cdot\nabla_h u_h\|_{L^2}\dd\tau+\int_0^t\|e^{\Delta_h(t-\tau)}u_3\cdot\p_3 u_h\|_{L^2}\dd\tau\\
 &:=J_{111}+J_{112}.
\end{aligned}
\end{equation*}
By \eqref{II2.1} and Minkowski's inequality,
\begin{equation*}
\begin{aligned}
 J_{111} &\leqslant\int_0^t\left\|\|e^{\Delta_h(t-\tau)}u_h\cdot\nabla_h u_h\|_{L^2_{x_h}}\right\|_{L^2_{x_3}}\dd\tau \\
 &\leqslant C \int_0^t\left\|   (t-\tau)^{-\frac{1}{2}}\|u_h\cdot\nabla_h u_h\|_{L^1_{x_h}}\right\|_{L^2_{x_3}}\dd\tau \\
 &\leqslant C \int_0^t   (t-\tau)^{-\frac{1}{2}}\left\|\|u_h\|_{L^2_{x_h}}\|\nabla_h u_h\|_{L^2_{x_h}}\right\|_{L^2_{x_3}}\dd\tau\\
 &\leqslant C \int_0^t  (t-\tau)^{-\frac{1}{2}} \|u_h\|_{L^\infty_{x_3}L^2_{x_h}}\|\nabla_h u_h\|_{L^2_{x_3}L^2_{x_h}}\dd\tau\\
 &\leqslant C \int_0^t  (t-\tau)^{-\frac{1}{2}} \|u_h\|_{L^2_{x_h}L^\infty_{x_3}}\|\nabla_h u_h\|_{ L^2 }\dd\tau\\
& \leqslant C \int_0^t   (t-\tau)^{-\frac{1}{2}} \|u_h\|^{\frac{1}{2}}_{L^2_{x_h}L^2_{x_3}}\|\p_3u_h\|^{\frac{1}{2}}_{L^2_{x_h}L^2_{x_3}}\|\nabla_h u_h\|_{ L^2 }\dd\tau\\
  &\leqslant C \int_0^t   (t-\tau)^{-\frac{1}{2}} \|u_h\|^{\frac{1}{2}}_{L^2}\|\p_3u_h\|^{\frac{1}{2}}_{L^2}\|\nabla_h u_h\|_{ L^2 }\dd\tau
\end{aligned}
\end{equation*}
Invoking \eqref{I4.6}-\eqref{I4.8}, we have
\begin{equation*}
\begin{aligned}
 J_{111} & \leqslant C C_0^2\varepsilon^2\int_0^t   (t-\tau)^{-\frac{1}{2}} (1+\tau)^{-\frac{\sigma}{4}+\frac{\delta}{2}}(1+\tau)^{-\frac{\sigma}{4}+\frac{3\delta}{2}}(1+\tau)^{-\frac{\sigma+1}{2}+\delta}\dd\tau\\
 &\leqslant C C_0^2\varepsilon^2\int_0^t (t-\tau)^{-\frac{1}{2}}(1+\tau)^{-\sigma-\frac{1}{2}+3\delta}\dd\tau\leqslant CC_0^2\varepsilon^2(1+t)^{-\frac{\sigma}{2}+\delta}(0\leqslant\delta\leqslant\frac{\sigma}{4})
\end{aligned}
\end{equation*}
After similar calculation as the term $J_{111}$, we can conclude that
\begin{equation*}
\begin{aligned}
 J_{112}  &\leqslant C \int_0^t  (t-\tau)^{-\frac{1}{2}} \|u_3\|_{L^\infty_{x_3}L^2_{x_h}}\|\p_3 u_h\|_{L^2_{x_3}L^2_{x_h}}\dd\tau\\
 &\leqslant C \int_0^t  (t-\tau)^{-\frac{1}{2}} \|u_3\|_{L^2_{x_h}L^\infty_{x_3}}\|\p_3 u_h\|_{ L^2 }\dd\tau\\
 &\leqslant C \int_0^t   (t-\tau)^{-\frac{1}{2}} \|u_3\|^{\frac{1}{2}}_{L^2_{x_h}L^2_{x_3}}\|\p_3u_3\|^{\frac{1}{2}}_{L^2_{x_h}L^2_{x_3}}\|\p_3 u_h\|_{ L^2 }\dd\tau\\
  &\leqslant C \int_0^t   (t-\tau)^{-\frac{1}{2}} \|u_3\|^{\frac{1}{2}}_{L^2}\|\nabla_h\cdot u_h\|^{\frac{1}{2}}_{L^2}\|\p_3 u_h\|_{ L^2 }\dd\tau\\
   & \leqslant C C_0^2\varepsilon^2\int_0^t   (t-\tau)^{-\frac{1}{2}} (1+\tau)^{-\frac{\sigma}{4}+\frac{\delta}{2}}(1+\tau)^{-\frac{\sigma}{2}+ 3\delta}(1+\tau)^{-\frac{\sigma+1}{4}+\frac{\delta}{2}}\dd\tau\\
 &\leqslant C C_0^2\varepsilon^2\int_0^t (t-\tau)^{-\frac{1}{2}}(1+\tau)^{-\sigma-\frac{1}{4}+4\delta}\dd\tau\leqslant CC_0^2\varepsilon^2(1+t)^{-\frac{\sigma}{2}+\delta}\,\, (0\leqslant\delta\leqslant\frac{\sigma}{6}-\frac{1}{12}).
\end{aligned}
\end{equation*}
This proves that
\begin{equation}\label{J4.21}
J_{11}\leqslant CC_0^2\varepsilon^2(1+t)^{-\frac{\sigma}{2}+\delta}\,\,  (0\leqslant\delta\leqslant\frac{\sigma}{6}-\frac{1}{12}).
\end{equation}
Since \begin{equation*}
\begin{aligned}
 J_{12} &= \int_0^t\|e^{\lambda_1 (t-\tau)}\mathcal{F}_s\widehat{(u\cdot\nabla u_3)}\|_{L^2}\dd\tau\\
 &=\int_0^t\|e^{\Delta_h(t-\tau)}u_h\cdot\nabla_h u_3\|_{L^2}\dd\tau+\int_0^t\|e^{\Delta_h(t-\tau)}u_3\cdot\p_3 u_3\|_{L^2}\dd\tau\\
 &=\int_0^t\|e^{\Delta_h(t-\tau)}u_h\cdot\nabla_h u_3\|_{L^2}\dd\tau+\int_0^t\|e^{\Delta_h(t-\tau)}u_3\cdot(\nabla_h\cdot u_h)\|_{L^2}\dd\tau\\
 &\leqslant \int_0^t (t-\tau)^{^{-\frac{1}{2}}} \left(\|u_h\|_{L^\infty_{x_3}L^2_{x_h}}\|\nabla_h u_3\|_{L^2_{x_3}L^2_{x_h}}+ \|u_3\|_{L^\infty_{x_3}L^2_{x_h}}\|\nabla_h\cdot u_h\|_{L^2_{x_3}L^2_{x_h}}\right)\dd\tau,\\
\end{aligned}
\end{equation*}
thus
\begin{equation*}\begin{aligned}
 J_{12}
 &\leqslant C \int_0^t   (t-\tau)^{-\frac{1}{2}} \|u_h\|^{\frac{1}{2}}_{L^2}\|\p_3u_h\|^{\frac{1}{2}}_{L^2}\|\nabla_h u_3\|_{ L^2 }\dd\tau+ \int_0^t   (t-\tau)^{-\frac{1}{2}} \|u_3\|^{\frac{1}{2}}_{L^2}\|\nabla\cdot u_h\|^{\frac{3}{2}}_{L^2} \dd\tau\\
&\leqslant C C_0^2\varepsilon^2\int_0^t (t-\tau)^{-\frac{1}{2}}(1+\tau)^{-\sigma-\frac{1}{2}+3\delta}\dd\tau+C C_0^2\varepsilon^2\int_0^t (t-\tau)^{-\frac{1}{2}}(1+\tau)^{-\sigma-\frac{3}{4}+2\delta}\dd\tau\\
&\leqslant CC_0^2\varepsilon^2(1+t)^{-\frac{\sigma}{2}+\delta}\,\,  (0\leqslant\delta\leqslant\frac{\sigma}{4}).
\end{aligned}
\end{equation*}
This completes the estimate of $J_1$, i.e.
$$\|J_1\|_{L^2}\leqslant CC_0^2\varepsilon^2(1+t)^{-\frac{\sigma}{2}+\delta}, \,\,(0\leqslant\delta\leqslant\frac{\sigma}{6}-\frac{1}{12}). $$
Now check the estimate of $J_2+J_6$, actually, we have
\begin{equation}\label{E4.22}
\begin{aligned}
 J_{2}+J_6 &\leqslant \int_0^t e^{\lambda_1(t-\tau)}\mathcal{F}_s(\widehat{\mathbb{P}u\cdot\nabla u})_3(\tau)\dd+  \int_0^t \frac{|\xi_3|}{|\xi_h|} e^{\lambda_1(t-\tau)}\mathcal{F}_s(\widehat{\mathbb{P}u\cdot\nabla u})_3(\tau)\dd \tau
\end{aligned}
\end{equation}
Substituting \eqref{E4.20} into above inequality, we obtain that
\begin{equation*}
\begin{aligned}
 \|J_{2}\|_{L^2}+\|J_6\|_{L^2} &\leqslant \int_0^t\left\|e^{\lambda_1(t-\tau)}\left(\frac{|\xi_h|^2}{|\xi|^2}\mathcal{F}_s(\widehat{u\cdot\nabla u_3})\right)\right\|_{L^2}\dd\tau+\int_0^t\left\| e^{\lambda_1(t-\tau)}\frac{ \ri \xi_3\xi_h}{|\xi|^2}\cdot\mathcal{F}_c(\widehat{u\cdot\nabla u_h})\right\|_{L^2}\dd \tau\\
 &+\int_0^t \left\|\left(\frac{|\xi_3|}{|\xi_h|} e^{\lambda_1(t-\tau)} \frac{|\xi_h|^2}{|\xi|^2}\mathcal{F}_s(\widehat{u\cdot\nabla u_3})\right)\right\|_{L^2}\dd \tau+\int_0^t  \left\|e^{\lambda_1(t-\tau)}\left(\frac{ \ri \xi_3\xi_h}{|\xi|^2}\cdot\mathcal{F}_c(\widehat{u\cdot\nabla u_h})\right)\right\|_{L^2}\dd \tau\\
 &\leqslant 2\int_0^t \|e^{\lambda_1(t-\tau)}\mathcal{F}_s(\widehat{u\cdot\nabla u_3})\|_{L^2}\dd \tau+2\int_0^t\|e^{\lambda_1(t-\tau)}\mathcal{F}_c(\widehat{u\cdot\nabla u_h})\|_{L^2}\dd \tau\\
 &=J_{12}+J_{11}\leqslant CC_0^2\varepsilon^2(1+t)^{-\frac{\sigma}{2}+\delta}, \,\,(0\leqslant\delta\leqslant\frac{\sigma}{6}-\frac{1}{12}).
\end{aligned}
\end{equation*}
Similarly, we can calculate that
\begin{equation*}
\begin{aligned}
 \|J_{3}\|_{L^2}+\|J_5\|_{L^2}+\|J_7\|_{L^2} &\leqslant  \int_0^t\left\|e^{\lambda_1(t-\tau)}\frac{\ri \xi_h\xi_3}{|\xi||\xi_h|}\sin\frac{|\xi_h|}{|\xi|}(t-\tau) \mathcal{F}_s\widehat{u\cdot\nabla\theta}(\tau)\right\|_{L^2}\dd\tau\\
 &+\int_0^t \left\|e^{\lambda_1(t-\tau)}\frac{|\xi_h|}{|\xi|}\sin\frac{|\xi_h|}{|\xi|}(t-\tau)\mathcal{F}_s\widehat{u\cdot\nabla\theta}(\tau)\right\|_{L^2}\dd \tau\\
 &+\int_0^t\left\|e^{\lambda_1(t-\tau)}\cos\frac{|\xi_h|}{|\xi|}(t-\tau) \mathcal{F}_s\widehat{u\cdot\nabla\theta}(\tau)\right\|_{L^2}\dd \tau\\
 &\leqslant 3\int_0^t\left\|e^{\lambda_1(t-\tau)}\mathcal{F}_s\widehat{u\cdot\nabla\theta}(\tau)\right\|_{L^2}\dd \tau
\end{aligned}
\end{equation*}
By the same procedures for $\|J_{111}\|_{L^2}$ and $\|J_{112}\|_{L^2}$, it is easy to see that
\begin{equation*}
\begin{aligned}
 \|J_{3}\|_{L^2}+\|J_5\|_{L^2}+\|J_7\|_{L^2} &\leqslant  CC_0^2\varepsilon^2(1+t)^{-\frac{\sigma}{2}+\delta}\,\,  (0\leqslant\delta\leqslant\frac{\sigma}{6}-\frac{1}{12}).
\end{aligned}
\end{equation*}
Using the argument to dealt with the formula stated in \eqref{E4.22}, we can observe that
\begin{equation*}
\begin{aligned}
 \|J_{4}\|_{L^2}   \leqslant  CC_0^2\varepsilon^2(1+t)^{-\frac{\sigma}{2}+\delta}\,\,  (0\leqslant\delta\leqslant\frac{\sigma}{6}-\frac{1}{12}).
\end{aligned}
\end{equation*}
Therefor, when $\varepsilon>0$ small enough, such that $3CC_0\varepsilon<\frac{1}{2}$,
$$\|u\|_{L^2}+\|\theta\|_{L^2}\leqslant 3CC^2_0\varepsilon^2(1+t)^{-\frac{\sigma}{2}+\delta}\leqslant \frac{C_0}{2}\varepsilon(1+t)^{-\frac{\sigma}{2}+\delta}\,\,  (0\leqslant\delta\leqslant\frac{\sigma}{6}-\frac{1}{12}).$$
\subsection{Estimates of $\|\p_3u\|_{L^2},\|\p_3\theta\|_{L^2}$ and verification of \eqref{I4.12}}

In this subsection, we prove the upper bounds for $\|\p_3u\|_{L^2}$ and $\|\p_3\theta\|_{L^2}$ and verify \eqref{I4.12}. We again make use of the integral representations \eqref{E4.14}, \eqref{E4.15} and \eqref{E4.16}. By  Lemma \ref{lem1} (4), we know that
\begin{equation}\label{E4.23}
\|\p_3 u_h\|_{L^2}=\frac{1}{\sqrt{2\pi^3}}\|\xi_3\mathcal{F}_c \widehat{u_h}\|_{L^2},\,\, \|\p_3 u_3\|_{L^2}=\|\nabla\cdot u_h\|_{L^2},\, \, \|\p_3 \theta\|_{L^2}=\frac{1}{\sqrt{2\pi^3}}\|\xi_3\mathcal{F}_s \hat{\theta}\|_{L^2}
\end{equation}
From \eqref{E4.14}, \eqref{E4.15}, \eqref{E4.16}, \eqref{E4.23} and Proposition \ref{prop2.1} , it is sufficient to consider the following terms one by one
\begin{eqnarray*}
&\xi_3 J_1&=\int_0^te^{\lambda_1(t-\tau)}\xi_3\mathcal{F}_c(\widehat{\mathbb{P}u\cdot\nabla u})_h(\tau)\dd\tau,\\
&\xi_3J_2&=\int_0^t\xi_3\frac{\ri \xi_h\xi_3}{|\xi_h|^2}\left(1-\cos\frac{|\xi_h|}{|\xi|}(t-\tau)\right)e^{\lambda_1(t-\tau)}\mathcal{F}_s(\widehat{\mathbb{P}u\cdot\nabla u})_3(\tau)\dd \tau,\\
&\xi_3J_3&=\int_0^t\xi_3\frac{\ri \xi_h\xi_3}{|\xi||\xi_h|}\left(\sin\frac{|\xi_h|}{|\xi|}(t-\tau)\right) e^{\lambda_1(t-\tau)} \mathcal{F}_s\widehat{u\cdot\nabla\theta}(\tau)\dd \tau,\\
&\xi_3J_4&=\int_0^t\xi_3\left(\cos\frac{|\xi_h|}{|\xi|}(t-\tau)\right) e^{\lambda_1(t-\tau)}\mathcal{F}_s(\widehat{\mathbb{P}u\cdot\nabla u})_3(\tau)\dd \tau, \\
&\xi_3J_5&=\int_0^t\xi_3\left(\frac{|\xi_h|}{|\xi|}\sin\frac{|\xi_h|}{|\xi|}(t-\tau)\right)e^{\lambda_1(t-\tau)}\mathcal{F}_s\widehat{u\cdot\nabla\theta}(\tau)\dd \tau,\\
&\xi_3J_6&=\int_0^t\xi_3\left(\frac{|\xi|}{|\xi_h|}\sin\frac{|\xi_h|}{|\xi|}(t-\tau)\right)e^{\lambda_1(t-\tau)}\mathcal{F}_s(\widehat{\mathbb{P}u\cdot\nabla u})_3(\tau)\dd \tau,\\
&\xi_3J_7&=\int_0^t\xi_3\left(\cos\frac{|\xi_h|}{|\xi|}(t-\tau)\right) e^{\lambda_1(t-\tau)}\mathcal{F}_s\widehat{u\cdot\nabla\theta}(\tau)\dd \tau.
\end{eqnarray*}
To check  $\|\xi_3J_1\|_{L^2}$, From \eqref{E4.19} we know that
\begin{equation*}
\begin{aligned}
\|\xi_3J_1\|_{L^2}&\leqslant\int_0^t\left\|e^{\lambda_1(t-\tau)}\left[\mathcal{F}_s(\widehat{\p_3(u\cdot\nabla u_h}))+\sum\limits_{k=1}^2\frac{ \xi_h\xi_k\xi_3}{|\xi|^2}\mathcal{F}_c(\widehat{u\cdot\nabla u_k})+\frac{\ri\xi_h\xi^2_3}{|\xi|^2}\mathcal{F}_s(\widehat{u\cdot\nabla u_3})\right]\right\|_{L^2}\dd\tau\\
&\leqslant \int_0^t\|e^{\lambda_1(t-\tau)}\mathcal{F}_s(\widehat{\p_3(u\cdot\nabla u_h}))\|_{L^2}\dd\tau+\int_0^t\|e^{\lambda_1(t-\tau)} |\xi_h|\mathcal{F}_c(\widehat{u\cdot\nabla u_h})\|_{L^2}\dd\tau+ \int_0^t\|e^{\lambda_1(t-\tau)} |\xi_h| \mathcal{F}_s(\widehat{u\cdot\nabla u_3})\|_{L^2}\dd\tau\\
&:=M_{11}+M_{12}+M_{13}.
\end{aligned}
\end{equation*}
We find that $M_{11}$ is subtle difficult. To end it, we can divide into three terms  as follows
\begin{equation*}
\begin{aligned}
M_{11}&\leqslant \int_0^t\|e^{\lambda_1(t-\tau)}\mathcal{F}_s(\widehat{\p_h(u\cdot\p_3 u_h}))\|_{L^2}\dd\tau+\int_0^t\|e^{\lambda_1(t-\tau)}\mathcal{F}_s(\widehat{\nabla_h\cdot u_h\cdot\p_3 u_h}))\|_{L^2}\dd\tau+\int_0^t\|e^{\lambda_1(t-\tau)}\mathcal{F}_s(\widehat{ u_3\cdot\p_3^2 u_h}))\|_{L^2}\dd\tau\\
&:=L_{11}+L_{12}+L_{13}
\end{aligned}
\end{equation*}
From Lemma \ref{lem3}, one refers that
\begin{equation*}
\begin{aligned}
L_{11}&\leqslant \int_0^t\|e^{\lambda_1(t-\tau)}\mathcal{F}_s(\widehat{\p_h(u_h\cdot\p_3 u_h}))\|_{L^2}\dd\tau \\
&\leqslant\int_0^t\|e^{\Delta_h(t-\tau)}\p_h(u_h\cdot\p_3 u_h)\|_{L^2}\dd\tau\\
&\leqslant  C\int_0^t(t-\tau)^{-\frac{3}{4}}\left\| \|u\cdot\p_3 u_h\|_{L^\frac{4}{3}_{x_h}}\right\|_{L^2_{x_3}}\dd\tau\\
&\leqslant C\int_0^t(t-\tau)^{-\frac{3}{4}} \| u\|_{L^2}^\frac{1}{4}\|\p_3 u_h\|_{L^2}^\frac{1}{4}\|\nabla_h u_h\|^\frac{1}{4}_{L^2}\|\p_3\nabla_hu_h\|_{L^2}^\frac{1}{4}\|\p_3 u_h\|_{L^2}\dd\tau.
\end{aligned}
\end{equation*}
By   \eqref{I4.6}-\eqref{I4.8}, we conclude that
\begin{equation*}
\begin{aligned}
L_{11} &\leqslant CC_0^2\varepsilon^2\int_0^t(t-\tau)^{-\frac{3}{4}}(1+\tau)^{-\frac{7\sigma}{8}+\frac{17}{4}\delta-\frac{1}{8}}\dd\tau\\
&\leqslant CC_0^2\varepsilon^2 (1+t)^{-\frac{\sigma}{2}+3\delta}, \left(0\leqslant\delta\leqslant\frac{3\sigma}{10}-\frac{1}{10}\right).
\end{aligned}
\end{equation*}
It is easy to see that
\begin{equation*}
\begin{aligned}
L_{12}&\leqslant \int_0^t\|e^{\lambda_1(t-\tau)}\mathcal{F}_s(\widehat{\nabla_h\cdot u_h\cdot\p_3 u_h}))\|_{L^2}\dd\tau \\
&\leqslant\int_0^t\|e^{\Delta_h(t-\tau)}\nabla_h\cdot u_h\cdot\p_3 u_h\|_{L^2}\dd\tau\\
&\leqslant  C\int_0^t(t-\tau)^{-\frac{1}{2}}\left\| \|\nabla_h  u_h\cdot\p_3 u_h\|_{L^1_{x_h}}\right\|_{L^2_{x_3}}\dd\tau\\
&\leqslant C\int_0^t(t-\tau)^{-\frac{1}{2}} \| \p_3 u_h\|_{L^2}^\frac{1}{2}\|\p^2_3 u_h\|_{L^2}^\frac{1}{2}\|\nabla_h u_h\|_{L^2} \dd\tau\\
&\leqslant C \int_0^t(t-\tau)^{-\frac{1}{2}} \| \p_3 u_h\|_{L^2}^\frac{3}{4}\|\p^3_3 u_h\|_{L^2}^\frac{1}{4}\|\nabla_h u_h\|_{L^2} \dd\tau,
\end{aligned}
\end{equation*}
here, the interpolation inequality $\|\p^2_3 u_h\|_{L^2}\leqslant C \| \p_3 u_h\|_{L^2}^\frac{1}{2}\|\p^3_3 u_h\|_{L^2}^\frac{1}{2}$ is used.

Taking advantage of  \eqref{I4.6}-\eqref{I4.8} again, we have got
\begin{equation*}
\begin{aligned}
L_{12}&\leqslant CC_0^2\varepsilon^2\int_0^t(t-\tau)^{-\frac{1}{2}}(1+\tau)^{-\frac{7}{8}\sigma+\frac{13}{4}\delta-\frac{1}{2}}\dd\tau\\
&\leqslant CC_0^2\varepsilon^2 (1+t)^{-\frac{\sigma}{2}+3\delta}, \left(0\leqslant\delta\leqslant\frac{3\sigma}{2} \right).
\end{aligned}
\end{equation*}
Let us compute the term $L_{13}$, in fact
 \begin{equation*}
 \begin{aligned}
L_{13}&\leqslant \int_0^t\|e^{\Delta_h(t-\tau)}u_3\cdot\p^2_3 u_h\|_{L^2}\dd\tau\\
&\leqslant C\int_0^t(t-\tau)^{-\frac{1}{2}}\left\|\|u_3\cdot\p^2_3 u_h\|_{L^1_{x_h}}\right\|_{L^2_{x_3}}\dd\tau\\
&\leqslant C\int_0^t(t-\tau)^{-\frac{1}{2}}\left\|\|u_3\|_{L^2_{x_h}}\|\p^2_3 u_h\|_{L^2_{x_h}}\right\|_{L^2_{x_3}}\dd\tau\\
&\leqslant C\int_0^t(t-\tau)^{-\frac{1}{2}} \|u_3\|^\frac{1}{2}_{L^2}\|\nabla_h\cdot u_h\|_{L^2}^\frac{1}{2}\|\p^2_3 u_h\| _{L^2}\dd\tau\\
&\leqslant C\int_0^t(t-\tau)^{-\frac{1}{2}} \|u_3\|^\frac{1}{2}_{L^2}\|\nabla_h\cdot u_h\|_{L^2}^\frac{1}{2}\|\p_3 u_h\|^\frac{1}{2} _{L^2}\|\p_3^3u_h\|_{L^2}^\frac{1}{2}\dd\tau
\end{aligned}
\end{equation*}
where we have applied the interpolation inequality $\|\p^2_3 u_h\|_{L^2}\leqslant C \| \p_3 u_h\|_{L^2}^\frac{1}{2}\|\p^3_3 u_h\|_{L^2}^\frac{1}{2}$.
We invoke the ansatz in \eqref{I4.6}-\eqref{I4.8} to obtain,
\begin{equation*}
 \begin{aligned}
L_{13} &\leqslant \int_0^t(t-\tau)^{-\frac{1}{2}} \|u_3\|^\frac{1}{2}_{L^2}\|\nabla_h\cdot u_h\|_{L^2}^\frac{1}{2}\|\p_3 u_h\|^\frac{1}{2} _{L^2}\|\p_3^3u_h\|_{L^2}^\frac{1}{2}\dd\tau\\
&\leqslant CC_0^2\varepsilon^2\int_0^t (t-\tau)^{-\frac{1}{2}}(1+\tau)^{-\frac{3}{4}\sigma-\frac{1}{4}+\frac{5}{2}\delta}\dd\tau\\
&\leqslant CC_0^2\varepsilon^2(1+t)^{-\frac{\sigma}{2}+3\delta}\left(\frac{1}{2}-\frac{\sigma}{2}\leqslant\delta<1\right),
\end{aligned}
\end{equation*}
hence, we infer that
$$M_{11}\leqslant CC_0^2\varepsilon^2(1+t)^{-\frac{\sigma}{2}+3\delta}\left(\frac{1}{2}-\frac{\sigma}{2}\leqslant\delta<1\right).$$

It is not difficult to observe that the estimates for $M_{12}$ and $M_{13}$ are very similar to those of $L_{11}$, thus we conclude that
\begin{equation*}
 \begin{aligned}
M_{12}+M_{13}\leqslant CC_0^2\varepsilon^2(1+t)^{-\frac{\sigma}{2}+3\delta},\left(0\leqslant\delta\leqslant\frac{3\sigma}{10}-\frac{1}{10}\right).
\end{aligned}
\end{equation*}
By \eqref{E4.20},we can estimate the term $\xi_3J_2$ as
\begin{equation*}
\begin{aligned}
\|\xi_3J_2\|_{L^2}&\leqslant\int_0^t\left\|\xi_3\frac{\ri \xi_h\xi_3}{|\xi_h|^2}\left(1-\cos\frac{|\xi_h|}{|\xi|}(t-\tau)\right)e^{\lambda_1(t-\tau)}\mathcal{F}_s(\widehat{\mathbb{P}u\cdot\nabla u})_3(\tau)\right\|_{L^2}\dd \tau\\
&\leqslant 2\int_0^t\left\| \frac{\ri \xi_h\xi^2_3}{|\xi_h|^2} e^{\lambda_1(t-\tau)}\mathcal{F}_s(\widehat{\mathbb{P}u\cdot\nabla u})_3(\tau)\right\|_{L^2}\dd \tau\\
&\leqslant C \int_0^t\left\| e^{\lambda_1(t-\tau)}\frac{ \ri \xi_h\xi^2_3}{|\xi_h|^2}\left (\frac{|\xi_h|^2}{|\xi|^2}\mathcal{F}_s(\widehat{u\cdot\nabla u_3})+\frac{ \ri \xi_3\xi_h}{|\xi|^2}\cdot\mathcal{F}_c(\widehat{u\cdot\nabla u_h})(\tau)\right)\right\|_{L^2}\dd \tau\\
&\leqslant C \int_0^t\left\| e^{\lambda_1(t-\tau)}\frac{  \ri \xi_h\xi^2_3}{|\xi_h|^2}\left (\frac{|\xi_h|^2}{|\xi|^2}\mathcal{F}_s(\widehat{u\cdot\nabla u_3})\right)\right\|_{L^2}\dd \tau+C \int_0^t\left\| e^{\lambda_1(t-\tau)}\frac{  \xi_h\xi^2_3}{|\xi_h|^2}\left(\frac{ \ri \xi_3\xi_h}{|\xi|^2}\cdot\mathcal{F}_c(\widehat{u\cdot\nabla u_h})(\tau)\right)\right\|_{L^2}\dd \tau\\
&\leqslant C\int_0^t\left\| e^{\lambda_1(t-\tau)}   \ri \xi_h \mathcal{F}_s(\widehat{u\cdot\nabla u_3}) \right\|_{L^2}\dd \tau+C \int_0^t\left\| e^{\lambda_1(t-\tau)}      \mathcal{F}_s(\widehat{\p_3u\cdot\nabla u_h})(\tau) \right\|_{L^2}\dd \tau\\
&= M_{13}+M_{11}.
\end{aligned}
\end{equation*}
This implies that, for $\frac{3}{4}\leqslant \sigma<1$
$$\|\xi_3J_2\|_{L^2}\leqslant M_{13}+M_{11}\leqslant CC_0^2\varepsilon^2(1+t)^{-\frac{\sigma}{2}+3\delta},\left(\frac{1}{2}-\frac{\sigma}{2}\leqslant\delta\leqslant\frac{3\sigma}{10}-\frac{1}{10}\right) $$

It is easy check that
\begin{equation*}
\begin{aligned}
\|\xi_3J_3\|_{L^2}&\leqslant\int_0^t\left\|\xi_3\frac{\ri \xi_h\xi_3}{|\xi||\xi_h|}\left(\sin\frac{|\xi_h|}{|\xi|}(t-\tau)\right) e^{\lambda_1(t-\tau)} \mathcal{F}_s\widehat{u\cdot\nabla\theta}(\tau)\right\|_{L^2}\dd \tau\leqslant \int_0^t\left\|  e^{\lambda_1(t-\tau)} \xi_3\mathcal{F}_s\widehat{u\cdot\nabla\theta}(\tau)\right\|_{L^2}\dd \tau\\
&\leqslant \int_0^t\left\|  e^{\lambda_1(t-\tau)}  \mathcal{F}_c\widehat{\p_3 (u\cdot\nabla\theta)}(\tau)\right\|_{L^2}\dd \tau \leqslant\int_0^t\left\|  e^{\Delta_h(t-\tau)}   \p_3 (u\cdot\nabla\theta) (\tau)\right\|_{L^2}\dd \tau~~(\because u\cdot\nabla\theta(x_h,0)=0)\\
&\leqslant\int_0^t\left\|  e^{\Delta_h(t-\tau)}   \p_3 u_h\cdot\nabla_h\theta (\tau)\right\|_{L^2}\dd \tau+ \int_0^t\left\|  e^{\Delta_h(t-\tau)}   \nabla_h\cdot u_h \p_3\theta  (\tau)\right\|_{L^2}\dd \tau\\
&+\int_0^t\left\|  e^{\Delta_h(t-\tau)}\nabla_h (u_h\cdot\p_3\theta) (\tau)\right\|_{L^2}\dd \tau+ \int_0^t\left\|  e^{\Delta_h(t-\tau)}   u_3\cdot\p^2_3\theta (\tau)\right\|_{L^2}\dd \tau\\
&:=M_{31}+M_{32}+M_{33}+M_{34}.
\end{aligned}
\end{equation*}

The same arguments, which is corresponding to those estimates of $L_{12},L_{11}$  and $L_{13}$, are used to the estimates of $M_{31}+M_{32}, M_{33},M_{34}$respectively. We can obtain that
\begin{equation*}
\begin{aligned}
\|\xi_3J_3\|_{L^2}&\leqslant CC_0^2\varepsilon^2(1+t)^{-\frac{\sigma}{2}+3\delta},\left(\frac{1}{2}-\frac{\sigma}{2}\leqslant\delta\leqslant\frac{3\sigma}{10}-\frac{1}{10}\right)
\end{aligned}
\end{equation*}
 Collecting these estimates and \eqref{E4.14}, one claims that $$\|\p_3 u_h\|_{L^2}\leqslant \frac{C_0}{2}\varepsilon(1+t)^{-\frac{\sigma}{2}+3\delta}.$$

 It is easy to find that there are the same bounds of $\|\xi_3J_7\|_{L^2}$  and $\|\xi_3J_3\|_{L^2}.$ Since
 \begin{equation*}
\begin{aligned}
\|\xi_3J_6\|_{L^2}&\leqslant \int_0^t\left\|\xi_3\left(\frac{|\xi|}{|\xi_h|}\sin\frac{|\xi_h|}{|\xi|}(t-\tau)\right)e^{\lambda_1(t-\tau)}\left[\frac{|\xi_h|^2}{|\xi|^2}\mathcal{F}_s(\widehat{u\cdot\nabla u_3})+\frac{ \ri \xi_3\xi_h}{|\xi|^2}\cdot\mathcal{F}_c(\widehat{u\cdot\nabla u_h})\right](\tau)\right\|_{L^2}\dd \tau\\
&\leqslant \int_0^t\|e^{\lambda_1(t-\tau)}|\xi_h| \mathcal{F}_s(\widehat{u\cdot\nabla u_3}\|_{L^2}\dd\tau+\int_0^t\| e^{\lambda_1(t-\tau)} |\xi_h| \mathcal{F}_c(\widehat{u\cdot\nabla u_h})(\tau)\|_{L^2}\dd \tau,\\
&=M_{13}+M_{12}
\end{aligned}
\end{equation*}
thus one has got
 \begin{equation*}
\begin{aligned}
\|\xi_3J_6\|_{L^2}&\leqslant CC_0^2\varepsilon^2(1+t)^{-\frac{\sigma}{2}+3\delta} \left(\frac{1}{2}-\frac{\sigma}{2}\leqslant\delta\leqslant\frac{3\sigma}{10}-\frac{1}{10}\right).
\end{aligned}
\end{equation*}
From the formula \eqref{E4.16}, and these estimates of $\|\xi_3J_6\|_{L^2}$ and $\|\xi_3J_7\|_{L^2}$, we obtain  for $\varepsilon>0$ small enough,
$$\|\p_3\theta\|_{L^2}\leqslant\frac{C_0}{2}\varepsilon (1+t)^{-\frac{\sigma}{2}+3\delta},\left(\frac{1}{2}-\frac{\sigma}{2}\leqslant\delta\leqslant\frac{3\sigma}{10}-\frac{1}{10}\right).$$

\subsection{Estimates of $\|\nabla_h u\|_{L^2}$ and $\|\nabla_h\theta\|_{L^2}$ and verification of \eqref{I4.13} }
 To verify the decay of the $\|\nabla_h u\|_{L^2}$ and $\|\nabla_h\theta\|_{L^2}$, we apply the techniques in the previous subsection. From Lemma \ref{lem1}, it suffices to check that $\|\ri\xi_h\mathcal{F} (V(t)v_0)\|_{L^2}.$  From the formula stated in Proposition \ref{prop4.1}, we shall verify the following scheme in $L^2-$norm
 \begin{eqnarray*}
&\ri\xi_h J_1&=\int_0^te^{\lambda_1(t-\tau)}\ri\xi_h\mathcal{F}_c(\widehat{\mathbb{P}u\cdot\nabla u})_h(\tau)\dd\tau,\\
&\ri\xi_hJ_2&=\int_0^t\ri\xi_h\frac{\ri \xi_h\xi_3}{|\xi_h|^2}\left(1-\cos\frac{|\xi_h|}{|\xi|}(t-\tau)\right)e^{\lambda_1(t-\tau)}\mathcal{F}_s(\widehat{\mathbb{P}u\cdot\nabla u})_3(\tau)\dd \tau,\\
&\ri\xi_hJ_3&=\int_0^t\ri\xi_h\frac{\ri \xi_h\xi_3}{|\xi||\xi_h|}\left(\sin\frac{|\xi_h|}{|\xi|}(t-\tau)\right) e^{\lambda_1(t-\tau)} \mathcal{F}_s\widehat{u\cdot\nabla\theta}(\tau)\dd \tau,\\
&\ri\xi_hJ_4&=\int_0^t\ri\xi_h\left(\cos\frac{|\xi_h|}{|\xi|}(t-\tau)\right) e^{\lambda_1(t-\tau)}\mathcal{F}_s(\widehat{\mathbb{P}u\cdot\nabla u})_3(\tau)\dd \tau, \\
&\ri\xi_hJ_5&=\int_0^t\ri\xi_h\left(\frac{|\xi_h|}{|\xi|}\sin\frac{|\xi_h|}{|\xi|}(t-\tau)\right)e^{\lambda_1(t-\tau)}\mathcal{F}_s\widehat{u\cdot\nabla\theta}(\tau)\dd \tau,\\
&\ri\xi_hJ_6&=\int_0^t\ri\xi_h\left(\frac{|\xi|}{|\xi_h|}\sin\frac{|\xi_h|}{|\xi|}(t-\tau)\right)e^{\lambda_1(t-\tau)}\mathcal{F}_s(\widehat{\mathbb{P}u\cdot\nabla u})_3(\tau)\dd \tau,\\
&\ri\xi_hJ_7&=\int_0^t\ri\xi_h\left(\cos\frac{|\xi_h|}{|\xi|}(t-\tau)\right) e^{\lambda_1(t-\tau)}\mathcal{F}_s\widehat{u\cdot\nabla\theta}(\tau)\dd \tau.
\end{eqnarray*}
We firstly consider the terms $\|\ri\xi_hJ_3\|_{L^2},\|\ri\xi_hJ_5\|_{L^2}$ and $\|\ri\xi_hJ_7\|_{L^2}$, then
\begin{equation*}
\begin{aligned}
\|\ri\xi_hJ_3\|_{L^2}+\|\ri\xi_hJ_5\|_{L^2}+\|\ri\xi_hJ_7\|_{L^2}&\leqslant \int_0^t\|\ri\xi_h e^{\lambda_1(t-\tau)} \mathcal{F}_s\widehat{u\cdot\nabla\theta}(\tau)\|_{L^2}\dd \tau\\
&\leqslant C\int_0^t\|\nabla_h e^{\Delta_h(t-\tau)}   u\cdot\nabla\theta (\tau)\|_{L^2}\dd \tau\\
&\leqslant C\int_0^t\|\nabla_h e^{\Delta_h(t-\tau)}   u_h\cdot\nabla_h\theta (\tau)\|_{L^2}\dd \tau+  C\int_0^t\|\nabla_h e^{\Delta_h(t-\tau)}   u_3\cdot\p_3\theta (\tau)\|_{L^2}\dd \tau\\
&:=N_{11}+N_{12}.
\end{aligned}
\end{equation*}
By the proof of subsection 5.4 in \cite{JYWu}, we choose $q$ satisfying
$$q=\frac{2}{1+\sigma}.$$
As $\frac{3}{4}\leqslant\sigma<1,$ we have $1<q<2$, then by Lemma \ref{lem3},
$$N_{12}\leqslant C\int_0^t (t-\tau)^{-\frac{1+\sigma}{2}}\left\|\|u_3\p_3\theta\|_{L^q_{x_h}}\right\|_{L^2_{x_3}}\dd\tau.$$
Using the proof in \cite{JYWu} and \eqref{II2.1}, it is easy to obtain
$$\left\|\|u_3\p_3\theta\|_{L^q_{x_h}}\right\|_{L^2_{x_3}}\leqslant C\|\nabla_h\cdot u_h\|_{L^2}^\frac{1}{2}\|u_3\|_{L^2}^{\sigma-\frac{1}{2}}\|\nabla_h u_3\|_{L^2}^{1-\sigma}\|\p_3\theta\|_{L^2}.$$  Thus, for any $\frac{3}{4}\leqslant\sigma<1,$
\begin{equation*}
\begin{aligned}
N_{12}&\leqslant C\int_0^t (t-\tau)^{-\frac{1+\sigma}{2}}\left\|\|u_3\p_3\theta\|_{L^q_{x_h}}\right\|_{L^2_{x_3}}\dd\tau\\
&\leqslant C\int_0^t (t-\tau)^{-\frac{1+\sigma}{2}}\|\nabla_h\cdot u_h\|_{L^2}^\frac{1}{2}\|u_3\|_{L^2}^{\sigma-\frac{1}{2}}\|\nabla_h u_3\|_{L^2}^{1-\sigma}\|\p_3\theta\|_{L^2}\dd\tau\\
&\leqslant CC_0^2\varepsilon^2\int_0^t(t-\tau)^{-\frac{1+\sigma}{2}}(1+\tau)^{-\frac{3}{4}-\sigma+4\delta}\dd\tau\\
&\leqslant CC_0^2\varepsilon^2(1+t)^{-\frac{1+\sigma}{2}+\delta}~\left(0<\delta\leqslant\frac{\sigma}{3}-\frac{1}{12}\right).
\end{aligned}
\end{equation*}
As the same techniques are applied for the term $N_{12}$, we easily obtain
$$N_{11}\leqslant C\int_0^t(t-\tau)^{-\frac{1+\sigma}{2}}\|u_h\|_{L^2}^{\sigma-\frac{1}{2}}\|\nabla_h u_h\|_{L^2}^{\frac{3}{2}-\sigma}\|\nabla_h\theta\|_{L^2}\leqslant CC_0^2\varepsilon^2(1+t)^{-\frac{1+\sigma}{2}+\delta}~\left(0<\delta<\frac{\sigma}{2}+\frac{1}{4}\right).$$
Therefore, we conclude that
$$\|\ri\xi_hJ_3\|_{L^2}+\|\ri\xi_hJ_5\|_{L^2}+\|\ri\xi_hJ_7\|_{L^2}\leqslant CC_0^2\varepsilon^2(1+t)^{-\frac{1+\sigma}{2}+\delta}~\left(0<\delta\leqslant\frac{\sigma}{3}-\frac{1}{12}\right).$$
From \eqref{E4.19} and \eqref{E4.20} and by the previous calculation  for $\|\p_3u_h\|_{L^2}$ and Lemma \ref{lem3}, we can achieve that
\begin{equation*}
\begin{aligned}
\|\ri\xi_hJ_1\|_{L^2}+\|\ri\xi_hJ_2\|_{L^2}+\|\ri\xi_hJ_4\|_{L^2}+\|\ri\xi_hJ_6\|_{L^2}&\leqslant C \int_0^t\|\ri\xi_h e^{\lambda_1(t-\tau)} \mathcal{F}_c\widehat{u\cdot\nabla u_h}(\tau)\|_{L^2}\dd \tau\\
&+C\int_0^t\|\ri\xi_h e^{\lambda_1(t-\tau)} \mathcal{F}_s\widehat{u\cdot\nabla u_3}(\tau)\|_{L^2}\dd \tau\\
& \leqslant C \int_0^t(t-\tau)^{-\frac{1+\sigma}{2}}\left\|\| u\cdot\nabla u_h\|_{L^q_{x_h}}\right\|_{L^2_{x_3}}\dd \tau\\
&+C\int_0^t(t-\tau)^{-\frac{1+\sigma}{2}}\left\|\| u\cdot\nabla u_3\|_{L^q_{x_h}}\right\|_{L^2_{x_3}}\dd \tau\\
&:=N_{21}+N_{22}.
\end{aligned}
\end{equation*}
For the terms $N_{21}$ and $N_{22}$, we can write the following form
\begin{equation*}
\begin{aligned}
N_{21}+N_{22}&\leqslant C \int_0^t(t-\tau)^{-\frac{1+\sigma}{2}}\left\|\| u_h\cdot\nabla_h u_h\|_{L^q_{x_h}}\right\|_{L^2_{x_3}}\dd \tau+C\int_0^t(t-\tau)^{-\frac{1+\sigma}{2}}\left\|\| u_3\cdot\p_3 u_h\|_{L^q_{x_h}}\right\|_{L^2_{x_3}}\dd \tau\\
&+C\int_0^t(t-\tau)^{-\frac{1+\sigma}{2}}\left\|\| u_h\cdot\nabla_h u_3\|_{L^q_{x_h}}\right\|_{L^2_{x_3}}\dd \tau+C\int_0^t(t-\tau)^{-\frac{1+\sigma}{2}}\left\|\| u_3\cdot(\p_3u_3)\|_{L^q_{x_h}}\right\|_{L^2_{x_3}}\dd \tau\\
&\leqslant C \int_0^t(t-\tau)^{-\frac{1+\sigma}{2}}\left\|\| u\cdot\nabla_h u\|_{L^q_{x_h}}\right\|_{L^2_{x_3}}\dd \tau+C\int_0^t(t-\tau)^{-\frac{1+\sigma}{2}}\left\|\| u_3\cdot\p_3 u_h\|_{L^q_{x_h}}\right\|_{L^2_{x_3}}\dd \tau.
\end{aligned}
\end{equation*}
By the techniques for the terms $N_{11}$ and $N_{12}$, we easily have got
$$N_{21}+N_{22}\leqslant CC_0^2\varepsilon^2(1+t)^{-\frac{1+\sigma}{2}+\delta}~\left(0<\delta\leqslant\frac{\sigma}{3}-\frac{1}{12}\right).$$
Therefore, by the formula of solutions \eqref{E4.14}-\eqref{E4.16} we know that  when $\varepsilon>0$ small enough,
$$\|\nabla_h u\|_{L^2},\|\nabla_h\theta\|_{L^2}\leqslant\frac{C_0}{2}\varepsilon(1+t)^{-\frac{1+\sigma}{2}+\delta}~\left(0<\delta\leqslant\frac{\sigma}{3}-\frac{1}{12}\right)$$
and   $$\|\p_3 u_3\|_{L^2}=\|\nabla_h\cdot u_h\|_{L^2}\leqslant \frac{C_0}{2}\varepsilon(1+t)^{-\frac{\sigma+1}{2}+\delta} ~\left(0<\delta\leqslant\frac{\sigma}{3}-\frac{1}{12}\right). $$
Combining all estimates, we need to hunt for  $\delta$ satisfying $$\frac{1}{2}-\frac{\sigma}{2}\leqslant \delta< \frac{\sigma}{8}-\frac{1}{16}.$$
It is easy to see that,  as $\sigma> \frac{9}{10}$, $\delta$ can be attained.

\section*{Acknowledgments} W. Yang is supported by the National Natural Science Foundation of China (No.12061003,12561036), by the National Natural Science Foundation of Ningxia (2023AAC02044).  A. Zang is supported by the Construction project of first-class subjects in Ningxia higher education (NXYLXK2017B09), by the National Natural Science Foundation of China (Nos. 12261093), by  Jiangxi Provincial Natural Science Foundation  (No. 20224ACB201004).

\section*{Declarations}
 There are not any conflicts of interest. This paper is  ethics approval and supported by National Natural Science Foundation of China(Nos.12061003,12261093, 12061080), National Natural Science Foundation of Ningxia(No. 2023AAC02044) and Jiangxi Provincial Natural Science Foundation(No. 20224ACB201004).

 \section*{Data availability}
 Data sharing not applicable to this article as no datasets were generated or analysed  during the current study.

\end{document}